\documentclass[a4paper,11pt]{amsart}

\usepackage[english,french]{babel}
\usepackage{url}
\usepackage{xcolor}
\usepackage{soul}
\usepackage[all,cmtip]{xy}
\usepackage{graphicx}

\usepackage[latin1]{inputenc}

\usepackage{amsmath,amssymb,enumerate,mathrsfs}

\newtheorem{thm}{Theorem}[section]
\newtheorem{corollary}[thm]{Corollary}
\newtheorem{proposition}[thm]{Proposition}
\newtheorem{lemma}[thm]{Lemma}
\newtheorem{step}[thm]{Step}

\theoremstyle{definition}

\newtheorem{remark}[thm]{Remark}

\newcommand{\NE}{\operatorname{NE}}
\newcommand{\Pic}{\operatorname{Pic}}

\newcommand{\Ext}{\operatorname{Ext}}

\newcommand{\Exc}{\operatorname{Exc}}
\newcommand{\codim}{\operatorname{codim}}

\newcommand{\ph}{\varphi}
\newcommand{\w}{\widetilde}
\newcommand{\ma}{\mathcal}
\newcommand{\la}{\longrightarrow}
\newcommand{\ol}{\mathcal{O}}
\newcommand{\wi}{\widehat}
\newcommand{\pr}{\mathbb{P}}
\newcommand{\Q}{\mathbb{Q}}
\newcommand{\C}{\mathbb{C}}
\newcommand{\R}{\mathbb{R}}
\newcommand{\Z}{\mathbb{Z}}
\newcommand{\N}{\mathcal{N}_1}
\newcommand{\Nu}{\mathcal{N}^1}

\newcommand{\Hom}{\operatorname{Hom}}

\makeatletter
\@namedef{subjclassname@2020}{\textup{2020} Mathematics Subject Classification}
\makeatother

\title[Fano manifolds with $\delta=3$]{Fano manifolds with Lefschetz defect 3}
\date{\today}
\subjclass[2020]{14J45}
\author{C.\ Casagrande}
\address{Cinzia Casagrande:
 Universit\`a di Torino,
 Dipartimento di Matematica,
via Carlo Alberto 10,
 10123 Torino - Italy}
\email{cinzia.casagrande@unito.it}
\author{E.\ A.\ Romano}
\address{Eleonora A.\ Romano: Universit\`a di Genova, Dipartimento di Matematica, via Dodecaneso 35, 16146 Genova - Italy}
\email{eleonoraanna.romano@unige.it}
\author{S.\ A.\ Secci}
\address{Saverio A.\ Secci:
 Universit\`a di Torino,
 Dipartimento di Matematica,
via Carlo Alberto 10,
 10123 Torino - Italy}
\email{saverioandrea.secci@unito.it}

\calclayout 

\numberwithin{figure}{section}
\numberwithin{table}{section}
\numberwithin{equation}{section}
\begin{document}
\selectlanguage{english}


\begin{abstract}
  Let $X$ be a smooth, complex Fano variety,  and $\delta_X$ its Lefschetz defect. By \cite{codim}, if $\delta_X\geq 4$, then $X\cong S\times T$, where $\dim T=\dim X-2$. In this paper we prove a structure theorem for the case where $\delta_X=3$. We show that there exists a smooth Fano variety $T$ with $\dim T=\dim X-2$ such that $X$ is obtained from $T$ with two possible explicit constructions; in both cases there is a $\pr^2$-bundle $Z$ over $T$
  such that $X$ is the blow-up of $Z$ along three pairwise disjoint smooth, irreducible, codimension 2 subvarieties.
  Then we apply the structure theorem to Fano 4-folds, to the case where $X$ has Picard number 5, and to Fano varieties having  an elementary divisorial contraction sending a divisor to a curve. In particular we complete the classification of Fano 4-folds with $\delta_X=3$, started in \cite{delta3_4folds}.

  This version of the paper incorporates the published article  with its corrigendum \cite{delta3}, where a missing case in 
Prop.~\ref{newfamilies} is added.
\end{abstract}
\selectlanguage{english}
\maketitle
\section{Introduction}
\noindent Let $X$ be a smooth, complex Fano variety. The \emph{Lefschetz defect} $\delta_X$ is an invariant of $X$ which  depends on the Picard number of prime divisors in $X$.
More precisely, consider the vector space $\N(X)$ of 1-cycles in $X$, with real coefficients, up to numerical equivalence; its dimension is the Picard number $\rho_X$ of $X$.
Given a prime divisor $D$ in $X$, we consider the pushforward $\iota_*\colon\N(D)\to\N(X)$ induced by the inclusion $\iota\colon D\hookrightarrow X$, and  set $\N(D,X):=\iota_*(\N(D))\subseteq\N(X)$; finally:
$$\delta_X:=\max\bigl\{\codim\N(D,X)\,|\,D\text{ a prime divisor in }X\bigr\}.$$
The main property of the Lefschetz defect is the following.
\begin{thm}[\cite{codim}, Th.~3.3]\label{delta4}
Let $X$ be a smooth Fano variety with $\delta_X\geq 4$. Then $X\cong S\times T$, where $S$ is a del Pezzo surface with $\rho_S=\delta_X+1$. 
\end{thm}  
In this paper we consider Fano varieties $X$ with $\delta_X=3$. Even if such $X$ does not need to be a product, it turns out that it still has a very rigid and explicit structure.  More precisely, we show that there exist a smooth Fano variety $T$ with $\dim T=\dim X-2$, $\rho_T=\rho_X-4$, and $\delta_T\leq 3$,
and a $\pr^2$-bundle $Z$ over $T$,
such that $X$ is obtained by blowing-up $Z$ along three pairwise disjoint smooth, irreducible, codimension 2 subvarieties $S_1,S_2,S_3$. The $\pr^2$-bundle $\ph\colon Z\to T$ is the projectivization of a suitable decomposable vector bundle on $T$, and $S_2$ and $S_3$ are sections of $\ph$. Instead $\ph_{|S_1}\colon S_1\to T$ is finite of degree $1$ or $2$.

To state our results, we first describe two explicit ways to obtain $X$ from $T$; the two constructions differ in the degree of $S_1$ over $T$. Then we prove (Th.~\ref{main}) that every Fano variety $X$ with Lefschetz defect $3$ is obtained by one of these two constructions; this is our main result.
\begin{proposition}[Construction A]\label{A}
Let $T$ be a smooth Fano variety with $\delta_T\leq 3$, and $D_1,D_2,D_3$ divisors on $T$ such that $-K_T+D_i-D_j$ is ample for every $i,j\in\{1,2,3\}$. 
Set 
$$Z:=\pr_T\bigl(\ol(D_1)\oplus\ol(D_2)\oplus\ol(D_3)\bigr)\stackrel{\ph}{\la} T,$$
and let 
$S_i\subset Z$ be the section of $\ph$ corresponding to the projection onto the summand $\ol_T(D_i)$, for $i=1,2,3$.
Finally let $h\colon X\to Z$  be the blow-up of $S_1,S_2,S_3$, and $\sigma:=\ph\circ h\colon X\to T$.

Then $X$ is a smooth Fano variety with $\delta_X=3$, $\dim X=\dim T+2$, and $\rho_X=\rho_T+4$.
\end{proposition}
 \begin{proposition}[Construction B]\label{B}
Let $T$ be a smooth Fano variety with $\delta_T\leq 3$, and $N$ a divisor on $T$ such that  $N\not\equiv 0$ and both $-K_T+N$ and $-K_T- N$ are ample. 
Set 
$$Z:=\pr_T\bigl(\ol(N)\oplus\ol\oplus\ol\bigr)\stackrel{\ph}{\la} T,$$
and let 
$S_2,S_3\subset Z$ be the sections of $\ph$ corresponding to the projections onto the summands $\ol$. Let $H$ be a tautological divisor of $Z$, and
 assume that there exists a smooth, codimension $2$  subvariety $S_1\subset Z$ which is a  complete intersection of general elements in the linear systems $|{H}|$ and $|2{H}|$.
Finally let $h\colon X\to Z$  be the blow-up of $S_1,S_2,S_3$, and $\sigma:=\ph\circ h\colon X\to T$.

Then $X$ is a smooth Fano variety with $\delta_X=3$, $\dim X=\dim T+2$, and $\rho_X=\rho_T+4$.
\end{proposition}
Let us note that Constructions A and B give new, explicit ways to construct Fano varieties of any dimension.
\begin{thm}[Structure theorem]\label{main}
Every smooth Fano variety $X$ with
$\delta_X=3$ is obtained with Construction A  or B.
\end{thm}
When $\dim X=4$ and $\rho_X=5$, Th.~\ref{main} is proved in \cite{delta3_4folds}, and used to obtain the classification of Fano $4$-folds with $\delta_X=3$ and $\rho_X=5$. To treat the general case, we apply the same strategy as in \cite{delta3_4folds}, adapting the proof step by step. Our starting point is the existence of a fibration in del Pezzo surfaces $\sigma\colon X\to T$ after \cite[Th.~3.3]{codim}. We give an outline of the proof at the end of the Introduction.

In particular, given a Fano variety $X$ of dimension $n$ with $\delta_X=3$, by Th.~\ref{main} there exist a smooth Fano variety $T$
with $\dim T=n-2$, $\rho_T=\rho_X-4$, and $\delta_T\leq 3$, and a morphism $\sigma\colon X\to T$, such that $\sigma$ realizes $X$ as a blow-up of a $\pr^2$-bundle $Z$ over $T$, as in Construction A or B. We will denote $X$ by $X_A$ in the former case, and by $X_B$ in the latter.

Let us now describe some applications of the structure theorem.
\subsection*{Fano 4-folds} 
In dimension $4$, we complete the \emph{classification of Fano $4$-folds with Lefschetz defect $3$,} started in \cite{delta3_4folds};  note that all of them are rational, see \cite[Cor.~1.3]{M_R}.
We treat the last case left open, that is $\delta_X=3$ and $\rho_X=6$, in which we show that there are 11  families, among which 8 are toric;
we refer the reader to Section \ref{4folds}  for more details and for the description of the three non-toric families and their invariants. Note that the toric case is simpler, as toric Fano $4$-folds have been classified by Batyrev \cite{bat2,sato}. The final classification is as follows.
\begin{proposition}
  \label{classification}
  Let $X$ be a Fano $4$-fold with $\delta_X=3$. Then $5\leq \rho_X\leq 8$ and there are
  19 families for $X$, among which 14 are toric. More precisely:
  \begin{enumerate}[--]
  \item if $\rho_X=8$, then $X\cong F\times F$, $F$ the blow-up of $\pr^2$ at $3$ non-collinear points;
  \item if     $\rho_X=7$, then $X\cong F\times F'$, $F'$ the blow-up of $\pr^2$ at $2$ points;
  \item if     $\rho_X=6$, there are
    11 families for $X$, among which 8 are toric;
     \item  if     $\rho_X=5$, there are 6  families for $X$, among which 4 are toric.
    \end{enumerate}
  \end{proposition}
We also apply our results to the study of conic bundles $\eta\colon X\to Y$ where $X$ is a Fano $4$-fold and $\rho_X-\rho_Y\geq 3$, see Cor.~\ref{cor:conic_bundles}.
\subsection*{The case $\rho_X=5$}
The second case that we consider in  detail is that of Fano varieties  with Lefschetz defect $3$ and Picard number $5$, the minimal one, in Section \ref{last}. After Th.~\ref{main} there are a smooth Fano variety $T$ with $\rho_T=1$ and a morphism $\sigma\colon X\to T$ that realizes 
$X$ as a blow-up of a $\pr^2$-bundle over $T$ as in Construction A or B. 
We show that in this case $T$ and $\sigma$ are uniquely determined, and we  list explicitly the different $X$ that are obtained from a given $T$. To state the result,
let us fix some notation for the case  $\rho_T=1$: $\ol_T(1)$ is the ample generator
of $\Pic(T) \cong\Z$, and $\ol_T(a):=\ol_T(1)^{\otimes a}$ for every $a\in\Z$. Moreover $\ol_T(-K_T) \cong \ol_T( i_T )$, where $i_T$ is the index of $T$.
\begin{proposition}\label{rho5}
  Let $X$ be a smooth Fano variety with $\delta_X=3$ and $\rho_X=5$. Then $T$ and the morphism $\sigma\colon X\to T$ as in Construction A or B are uniquely determined; we have $X=X_A$ if $\sigma$ is smooth, and $X=X_B$ if $\sigma$ is not smooth.
  
\smallskip
  
Let $X=X_A$. Then there are uniquely determined integers $a,b$ with 
$$
b\leq 0,\quad |b| \leq \frac{i_T-1}{2},\quad\text{and}\quad |b| \leq a \leq i_T-1-|b|,
$$
such that $X$ is obtained with Construction A from $T$ with
$Z=\pr_T(\ol(a)\oplus\ol\oplus\ol(b))$.

\smallskip

Let $X=X_B$. Then there is a 
 uniquely determined integer $a$ with 
$1\leq a \leq i_T-1$
such that $X$ is obtained with Construction B from $T$ with
$Z=\pr_T(\ol(a)\oplus\ol\oplus\ol)$.
\end{proposition}
\subsection*{Fano varieties containing a divisor with Picard number 2}
We note that the assumptions $\rho_X=5$ and $\delta_X=3$ imply that $X$ contains a prime divisor $D$ with $\dim\N(D,X)=2$; in fact it is easy to see that all the varieties as in Prop.~\ref{rho5} also contain a prime divisor $D'$ with $\rho_{D'}=2$. We obtain the following application to Fano varieties containing a prime divisor with $\rho=2$; an analogous result for the case of a prime divisor with $\rho=1$ is given in \cite[Th.~3.8]{minimal}.
\begin{corollary}\label{rho2}
Let $X$ be a smooth Fano variety containing a prime divisor $D$ with $\rho_D=2$, or more generally with $\dim\N(D,X)=2$. Then either $X\cong S\times T$ where $S$ is a del Pezzo surface and $\rho_T=1$, or $\rho_X\leq 5$. Moreover, $\rho_X=5$ if and only if $X$ is as in Prop.~\ref{rho5}.
\end{corollary}
\subsection*{Elementary divisorial contractions}
Corollary \ref{rho2} is also related to the study of \emph{Fano varieties having an elementary divisorial contraction 
$\tau\colon X\to X'$ where $\tau(\Exc (\tau))$ is a curve,} because then
 automatically  $\dim\N(\Exc(\tau),X)=2$. It follows from Th.~\ref{delta4}  that $\rho_X\leq 5$, and Tsukioka \cite{toru5} has classified the case $\rho_X=5$ when $\tau$ is the blow-up of a smooth curve in a smooth variety.
We generalize this classification to an arbitrary elementary divisorial contraction $\tau$ such that $\dim\tau(\Exc(\tau))=1$, as follows.
\begin{thm}\label{n-1,1}
  Let $X$ be a smooth Fano variety of dimension $n\geq 4$ with $\rho_X=5$. Then the following are equivalent:
  \begin{enumerate}[$(i)$]
  \item there is an elementary divisorial contraction $\tau\colon X\to X'$ such that $\tau(\Exc(\tau))$ is a curve;
  \item  $X$ is obtained with Construction A or B from a smooth Fano variety $T$ such that $\dim T=n-2$, $\rho_T=1$,  $i_T>1$, with $Z=\pr_T(\ol(a)\oplus\ol\oplus\ol)$ and $a\in\{1,\dotsc,i_T-1\}$.
\end{enumerate}
   If these conditions hold we have $\Exc(\tau)\cong \pr^1\times T$, $\ma{N}_{\Exc(\tau)/X}\cong \pi_{\pr_1}^*\ol(-1)\otimes\pi_T^*\ol(-a)$, and $\tau(\Exc(\tau))\cong\pr^1$.
 \end{thm}
 When $\dim X=4$, Cor.~\ref{rho2} and Th.~\ref{n-1,1} are already proved in \cite{delta3_4folds}.
\subsection*{Strategy of the proof of Theorem \ref{main}}
Our starting point is the existence, from \cite[Th.~3.3]{codim}, of a flat fibration $\sigma\colon X\to T$, where $T$ is a smooth Fano variety with $\dim T=\dim X-2$ and $\rho_T=\rho_X-4$; moreover $\sigma$ factors as $X\to Y\to T$, where the first map is a conic bundle, and the second one is smooth with fiber $\pr^1$. We collect the properties of $\sigma$ in Th.~\ref{factorization}.

We show that the fibration $\sigma\colon X\to T$ has a different factorization as $X\to Z\stackrel{\ph}{\to} T$, where $\ph$ is a $\pr^2$-bundle, and $X\to Z$ is the blow-up of three pairwise disjoint smooth, irreducible subvarieties $S_i\subset Z$ of codimension $2$, horizontal for $\ph$. 

 Then we prove that  $S_2$ and $S_3$ are always sections of $\ph$.
When $S_1$ is a section as well, we show that $\sigma\colon X\to T$ is as in Construction A.

Otherwise, we study the degree $d>1$ of  $S_1$ over $T$. Fiberwise, for $t\in T$ general, we have 
$Z_t:=\ph^{-1}(t)\cong\pr^{2}$, $X_t:=\sigma^{-1}(t)$ a smooth del Pezzo surface, and $X_t\to Z_t$ the blow-up of the $d+2$ points $(S_1\cup S_2\cup S_3)\cap Z_t$.
This is the hardest part of the proof; we restrict to a general curve $C\subset T$, and construct  in $\overline{Z}:=\ph^{-1}(C)$ a divisor $D$ which is a $\pr^1$-bundle over $C$ and contains $S_1\cap \overline{Z}$. Therefore, 
for $t\in C$ general,
the $d$ points $S_1\cap Z_t$ are aligned; since $X_t$ is del Pezzo, this implies that $d=2$. Finally we show that  $\sigma\colon X\to T$ is as in Construction B.
\subsection*{Structure of the paper} In Section \ref{notation} we fix the notation and prove some preliminary results. In Sections \ref{first} and \ref{second} we present Constructions A and B respectively, and prove Propositions \ref{A} and \ref{B}, while in Section \ref{proof} we prove
Th.~\ref{main}.

In the second part of the paper we consider the applications of the structure theorem. In Section \ref{sec6} we study the del Pezzo fibration $\sigma\colon X\to T$; we describe its fibers, the relative cone $\NE(\sigma)$, the relative contractions, and the different factorizations of $\sigma$. We also give some conditions on $T$ in order to perform Constructions A and B to obtain 
 Fano varieties $X$  different from the product $F\times T$, $F$ the blow-up of $\pr^2$ at three non-collinear points. Finally in Section \ref{4folds} we give the applications to Fano $4$-folds, and in Section \ref{last} the applications to Fano varieties with $\rho_X=5$.
\section{Notation and preliminaries}\label{notation}
\noindent We work over the field of complex numbers. Let $X$ be a smooth projective variety of arbitrary dimension. 

 $\mathcal{N}_{1}(X)$ (respectively, $\mathcal{N}^{1}(X)$) is the real vector space of one-cycles (respectively, Cartier divisors) with real coefficients, modulo numerical equivalence, and
  $\dim \mathcal{N}_{1}(X)=\dim \mathcal{N}^{1}(X)=\rho_{X}$ is the Picard number of $X$.

  Let $C$ be a one-cycle of $X$, and $D$ a divisor of $X$. We denote by $[C]$ (respectively, $[D]$) the numerical equivalence class in $\mathcal{N}_{1}(X)$ (respectively, $\mathcal{N}^{1}(X)$). We also denote by $C^{\perp}\subset\Nu(X)$ (respectively $D^{\perp}\subset\N(X)$) the orthogonal hyperplanes.
 
  The symbol $\equiv$ stands for numerical equivalence (for both one-cycles and divisors), and $\sim$ stands for linear equivalence of divisors.

  $\operatorname{NE}(X)\subset \mathcal{N}_{1}(X)$ is the convex cone generated by classes of effective curves. An \emph{extremal ray} $R$ is a 
one-dimensional face of $\NE(X)$. When $X$ is Fano, the \emph{length} of $R$ is $\ell(R)=\min\{-K_X\cdot C\,|\,C\text{ a rational curve in }X\}$.

A facet of a convex polyhedral cone $\ma{C}$ is a face of codimension one; moreover $\ma{C}$ is simplicial when it can be generated by linearly independent elements. We denote by $\langle S\rangle$ the convex cone generated by a subset $S\subset \N(X)$. 

 A \textit{contraction} of $X$ is a surjective morphism $\varphi\colon X\to Y$ with connected fibers, where $Y$ is normal and projective.
 The \textit{relative cone} $\text{NE}(\varphi)$ of $\varphi$ is the convex subcone of $\text{NE}(X)$ generated by classes of curves contracted by $\varphi$. A contraction is \emph{elementary} if $\rho_X-\rho_Y=1$.

 A \textit{conic bundle} $X\to Y$ is a  contraction of fiber type where every fiber is one-dimensional and $-K_X$ is relatively ample.

\medskip

We gather here some preliminary results that we need in the sequel.
\begin{remark}\label{relative}
  Let $T$ be a smooth projective variety and $D_1,\dotsc,D_r$ divisors on $T$. Set $Z:=\pr_T(\ol(D_1)\oplus\cdots\oplus\ol(D_r))\stackrel{\ph}{\to} T$ and
  let  $H$ be a tautological divisor. 
  For every $i=1,\dotsc,r$ set
  $$F_i:=\pr_T\bigl(\oplus_{j\neq i}\ol(D_j)\bigr)\hookrightarrow Z$$
  with the embedding given by the projection $\oplus_{j}\ol_T(D_j)\twoheadrightarrow\oplus_{j\neq i}\ol_T(D_j)$. Then:
  $$F_i\sim H-\ph^* (D_i)\ \text{ for every }i=1,\dotsc,r,\quad \text{and }\ 
  K_Z-\ph^*K_T\sim -F_1-\cdots-F_r.$$
\end{remark}
\begin{lemma}\label{E}
Notation as in Rem.~\ref{relative}
with $r=3$. Let 
$S_3\subset Z$ be the section of $\ph$ corresponding to the projection
$\ol_T(D_1)\oplus\ol_T(D_2)\oplus\ol_T(D_3)\twoheadrightarrow\ol_T(D_3)$, and let $h\colon X\to Z$ be the blow-up of $S_3$, with exceptional divisor $E_3\subset X$.

 Then $E_3\cong \pr_T(\ol(-K_T+D_3-D_1)\oplus\ol(-K_T+D_3-D_2))$ and 
$(-K_X)_{|E_3}$
 is the tautological line bundle.
\end{lemma}
\begin{proof}
We have
$\ma{N}_{S_3/Z}^{\vee}\cong\ol_T(D_1-D_3)\oplus\ol_T(D_2-D_3)$, $E_3\cong\pr_T(\ma{N}_{S_3/Z}^{\vee})$, and $\ol_X(-E_3)_{|E_3}$ is the tautological line bundle.

Let  $E_i'\subset X$  be the transform of $F_i\subset Z$, for $i=1,2,3$; note that
$S_3=F_1\cap F_2$ and $S_3\cap F_3=\emptyset$. Set $\sigma:=\ph\circ h\colon X\to T$ and $\sigma_3:=\sigma_{|E_3}\colon E_3\to T$.
By Rem.~\ref{relative} we have $-K_Z\sim\ph^*(-K_T)+F_1+F_2+F_3$, so that
\begin{gather*}
-K_X\sim h^*(-K_Z)-E_3\sim\sigma^*(-K_T)+E_1'+E_2'+E_3'+E_3,\\
 \text{ and }\ \ol_X(-K_X)_{|E_3}\cong \ol_{E_3}\bigl(\sigma_3^*(-K_T)+G_1+G_2\bigr)\otimes\ol_X(E_3)_{|E_3}
\end{gather*}
where $G_i:=E_i'\cap E_3$ for $i=1,2$. Note that $G_1$ and $G_2$
 are the sections of $\sigma_3$
corresponding to the two summands of $\ol_T(D_1-D_3)\oplus\ol_T(D_2-D_3)$,
 so by Rem.~\ref{relative} we have $\ol_{E_3}(G_1+G_2)\cong \ol_X(-2E_3+\sigma^*(2D_3-D_1-D_2))_{|E_3}$ . Finally we get
$$ \ol_X(-K_X)_{|E_3}\cong
  \ol_X(-E_3)_{|E_3}\otimes\ol_{E_3}\bigl(\sigma_3^*(-K_T+2D_3-D_1-D_2)\bigr),
$$
which yields the statement.
\end{proof}
In a similar way one shows the following.
\begin{lemma}\label{E'}
  Notation as in Rem.~\ref{relative}
with $r=3$. For $i=2,3$
let 
$S_i\subset Z$ be the section of $\ph$ corresponding to the projection
$\ol_T(D_1)\oplus\ol_T(D_2)\oplus\ol_T(D_3)\twoheadrightarrow\ol_T(D_i)$, and let $h\colon X\to Z$ be the blow-up of $S_2$ and $S_3$. Let $E_1'\subset X$ be the transform of $F_1\subset Z$. 

 Then $E_1'\cong \pr_T(\ol(-K_T+D_2-D_1)\oplus\ol(-K_T+D_3-D_1))$ and 
$(-K_X)_{|E_1'}$
 is the tautological line bundle.
\end{lemma}
We will also need the following properties.
\begin{remark}\label{completeint}
  Let $X$ be a smooth projective variety. Then $\N(X)$ is generated, as a vector space, by classes of complete intersections of very ample divisors.
\end{remark}
\begin{proof}
Let $n$ be the dimension of $X$
  and $H$ a very ample divisor. By the hard Lefschetz theorem, the linear map
  \begin{gather*}
    \Nu(X)\la\N(X)\\
    [D]\mapsto [D\cdot H^{n-2}]
  \end{gather*}
is an isomorphism. Since the ample cone has maximal dimension in $\Nu(X)$, we can choose very ample divisors $H_1,\dotsc,H_{\rho_X}$ on $X$ such that their classes generate $\Nu(X)$; then their images $[H_1\cdot H^{n-2}],\dotsc,[H_{\rho_X}\cdot H^{n-2}]$ generate $\N(X)$. 
\end{proof}
\begin{lemma}\label{levico}
  Let $Y$ be a smooth Fano variety and $\xi\colon Y\to T$ a smooth morphism with fiber $\pr^1$.
  Let $A_1,A_2\subset Y$ be disjoint prime divisors such that $\xi$ is finite on $A_i$ for $i=1,2$.
  Then at least one among $A_1,A_2$ is a section of $\xi$.
  \end{lemma}
  \begin{proof}
    We first show the following claim:

     \medskip

    \emph{If    $[A_1]$ and $[A_2]$ are multiples in $\Nu(Y)$, then $Y\cong\pr^1\times T$ and $A_i=\{pt\}\times T$.}

    \medskip
    
Let $\lambda\in\Q_{>0}$ be such that $A_1\equiv\lambda A_2$. If $C\subset A_1$ is a curve, then $A_1\cdot C=\lambda A_2\cdot C=0$ because $A_1\cap A_2=\emptyset$. This shows that $A_1$ is nef;
let $\alpha\colon Y\to Y_0$ be the contraction 
such that $\NE(\alpha)=
A_{1}^{\perp}\cap\NE(Y)$.

We note that $\alpha(A_1)$ is a point because $\alpha(C)=\{pt\}$ for every curve $C\subset A_1$; on the other hand $[A_1]\in \alpha^*\Nu(Y_0)$, so that $\dim Y_0=1$ and $Y_0\cong\pr^1$.
Moreover $\alpha$ is finite on the fibers $F$ of $\xi\colon Y\to T$, because
$\xi$ is finite on $A_1$, hence $A_1\cdot F>0$. Then $Y\cong T\times\pr^1$ and $\alpha$ is the projection, by \cite[Lemma 4.9]{31}. This yields the claim.

\medskip

Set $d_i:=\deg\xi_{|A_i}$ for $i=1,2$, so that $d_2A_1-d_1A_2\equiv\xi^*(N_0)$ for some divisor $N_0$ on $T$.

   If $N_0\equiv 0$, then $d_2A_1\equiv d_1A_2$ and we can apply the claim.

If $N_0\not\equiv 0$, then $N_0^{\perp}\subset\N(T)$ is a hyperplane, 
and by Rem.~\ref{completeint} there exists
a curve $C\subset T$, complete intersection of very ample divisors 
$H_1,\dotsc,H_{n-3}$, such that 
 $N_0\cdot C\neq 0$. We choose   
 $H_i$ general in their linear system.

Set $S:=\xi^{-1}(C)\subset Y$ and $C_i:=A_{i\,|S}$ for $i=1,2$; then $C_1\cap C_2=\emptyset$ and by Bertini
$C,C_1,C_2$ are smooth irreducible curves, and
 $S$ is a smooth ruled surface.

Suppose that $N_0\cdot C>0$. Then
$$d_1(C_2)^2=d_1A_2\cdot C_2=\bigl(d_2A_1-\xi^*(N_0)\bigr)\cdot 
C_2=-\xi^*(N_0)\cdot 
C_2=-d_2N_0\cdot C,$$
so that
$(C_2)^2<0$. We deduce that
$C_2$ is the negative section of the ruled surface $S$,  hence $d_2=1$. Since  $\xi_{|A_2}$ is finite, $A_2$ is a section of $\xi$.

In the case $N_0\cdot C<0$, we conclude in a similar way that $A_1$ is a section.
\end{proof}
\section{Construction A}\label{first}
\noindent Let $n\in\Z_{\geq 2}$ and $T$  a smooth Fano variety of dimension $n-2$. We consider three divisors $D_1,D_2,D_3$ in $T$ 
and set 
$$Z:=\pr_T\bigl(\ol(D_1)\oplus\ol(D_2)\oplus\ol(D_3)\bigr)\stackrel{\ph}{\la} T,$$
so that $Z$ is a smooth projective variety of dimension $n$  with $\rho_Z=\rho_T+1$.
Let 
$S_i\subset Z$ be the section of $\ph$ corresponding to the projection
$\ol_T(D_1)\oplus\ol_T(D_2)\oplus\ol_T(D_3)\twoheadrightarrow\ol_T(D_i)$, for $i=1,2,3$. 
Finally let $h\colon X\to Z$  be the blow-up of $S_1,S_2,S_3$, so that $X$ is a smooth projective variety of dimension $n$  with $\rho_X=\rho_T+4$, and set $\sigma:=\ph\circ h\colon X\to T$.
\begin{lemma}\label{fanoA}
The following are equivalent:
\begin{enumerate}[$(i)$]
\item
$X$ is Fano;
\item
$-K_T+D_i-D_j$ is ample for every $i,j\in\{1,2,3\}$.
\end{enumerate}
\end{lemma}
\begin{proof}
For $i=1,2,3$ let $E_i\subset X$ be the exceptional divisor over $S_i$,  $F_i\subset Z$  as in Rem.~\ref{relative}, and
  $E_i'\subset X$ its transform.

  The implication $(i)\Rightarrow(ii)$ follows directly from Lemma \ref{E}. 

Let us show the converse
$(ii)\Rightarrow(i)$. 
We first show that $-K_X$ is strictly nef, namely that $-K_X\cdot\Gamma>0$ for every irreducible curve $\Gamma\subset X$.

Again, it follows directly from Lemmas \ref{E} and \ref{E'} that $(-K_X)_{|E_i}$ and $(-K_X)_{|E_i'}$ are ample divisors respectively on $E_i$ and $E_i'$, for every $i=1,2,3$. Therefore if $\Gamma$ is contained in one of these divisors, we have $-K_X\cdot\Gamma>0$.

Assume that $\sigma(\Gamma)$ is a point. By construction, the sections $S_i$ of $\ph$ are fibrewise in general linear position, so that every fiber $X_t:=\sigma^{-1}(t)$  is a smooth del Pezzo surface. We have $\Gamma\subset X_t$ for some $t\in T$, therefore $-K_X\cdot\Gamma=-K_{X_t}\cdot\Gamma>0$.

Now assume that $\Gamma$ is not contracted by $\sigma$, and that $\Gamma\not\subset E_i, E_i'$ for $i=1,2,3$.
It follows from Rem.~\ref{relative} that:
$$-K_X\sim\sigma^*(-K_T)+\sum_{i=1}^3(E_i+E_i').$$
Then $\sigma(\Gamma)$ is an irreducible curve in $T$, and  $-K_X\cdot\Gamma=-K_T\cdot\sigma_*(\Gamma) + \sum_{i=1}^3(E_i+E_i')\cdot\Gamma>0$.

\bigskip

Now we show that $-K_X$ is big. Consider 
the divisor $A_Z:=F_1+ \varphi^*(-K_T)$. By Rem.~\ref{relative}, $F_1+\ph^*(D_1)$ is a tautological divisor for $\pr_T(\ol(D_1)\oplus\ol(D_2)\oplus\ol(D_3))$, therefore $A_Z=F_1+\ph^*(D_1)+\ph^*(-K_T-D_1)$ is a tautological divisor for
$\pr_T(\ol(-K_T)\oplus\ol(-K_T+D_2-D_1)\oplus\ol(-K_T+D_3-D_1))$, hence
 $A_Z$
 is ample on $Z$.

 Therefore there  exist non-negative integers $m,c_i$ such that the divisor
 \begin{align*}
   A_X&:=m h^*(A_Z) - \sum_{i=1}^3c_iE_i\\
&=m\bigl(E_1'+E_2+E_3+\sigma^*(-K_T)\bigr)- \sum_{i=1}^3c_iE_i\\
&  =m\sigma^*(-K_T)+mE_1'-c_1E_1+(m-c_2)E_2+(m-c_3)E_3
   \end{align*}
is ample on $X$. Then
\begin{equation*}
\begin{split}
  -mK_X & \sim m\sigma^*(-K_T)+m\sum_{i=1}^3(E_i+E_i')\\
&  =A_X+(c_1+m)E_1+c_2E_2+c_3E_3
  +mE_2'+mE_3'.
\end{split}
\end{equation*}
 We have written a  multiple of $-K_X$ as the sum of an ample divisor and an effective divisor, which implies that $-K_X$ is big. 

Finally, as $-K_X$ is strictly nef and big, it is ample by the base point free theorem.
\end{proof}
\begin{lemma}\label{deltaA}
  Suppose that $X$ is Fano.  Then $\delta_X\geq 3$, and $\delta_X=3$ if and only if $\delta_T\leq 3$.
\end{lemma}
 \begin{proof}
   The divisor $E_1\subset X$ is a $\pr^1$-bundle over $T$, so that $\dim\N(E_1,X)\leq \rho_{E_1}=\rho_T+1=\rho_X-3$, which implies that $\delta_X\geq 3$.

   If $\delta_X=3$, then $\delta_T\leq \delta_X=3$ by \cite[Rem.~3.3.18]{codim}. If instead $\delta_X\geq 4$, then by Th.~\ref{delta4} we have $X\cong S\times X'$ where $S$ is a  del Pezzo surface with $\rho_S=\delta_X+1\geq 5$; moreover $\sigma$  must be a product morphism (see \cite[Lemma 2.10]{eleonora}). Since the fiber of $\sigma$ is a del Pezzo surface with $\rho=4$, we must have $T\cong S\times T'$. Then $\delta_T\geq\delta_S$ again by \cite[Rem.~3.3.18]{codim}; on the other hand since $S$ is a surface, it is easy to see that $\delta_S=\rho_S-1=\delta_X$, hence 
    $\delta_T\geq \delta_X\geq 4$. 
\end{proof}
Prop.~\ref{A} follows from Lemmas \ref{fanoA} and \ref{deltaA}.
\section{Construction B}\label{second}
\noindent Let $n\in\Z_{\geq 2}$,  $T$ a smooth Fano variety of dimension $n-2$, and $N$ a  divisor in $T$ such that $N\not\equiv 0$.
Set
$$Z:=\pr_T(\ol(N)\oplus\ol\oplus \ol)\stackrel{\ph}{\la}T$$ and let $H$ be a tautological divisor of $Z$. Note that $h^0(Z,H)=h^0(T,\ol(N)\oplus\ol\oplus \ol)\geq 2$, and in fact the
constant surjections $\ol\oplus\ol\twoheadrightarrow \ol$ yield a pencil of 
divisors $\pr_T(\ol(N)\oplus\ol)\hookrightarrow Z$ in the linear system $|H|$ (see Rem.~\ref{relative}).

We assume that a  complete intersection of general elements in the linear systems $|{H}|$ and $|2{H}|$
is smooth, and we fix such a complete intersection $S_1\subset Z$.

We also consider the divisor $D:=\pr_T(\ol\oplus\ol)\hookrightarrow Z$ given by the  projection
$\ol(N)\oplus\ol\oplus \ol\twoheadrightarrow \ol\oplus\ol$, so that $D\cong\pr^1\times T$ with $\ma{N}_{D/Z}\cong\pi_{\pr^1}^*(\ol_{\pr^1}(1))\otimes\pi_T^*(\ol_T(-N))$, and the two sections $S_2,S_3\subset D$,
$S_i\cong\{pt\}\times T\subset D$, again corresponding to the two projections  $\ol\oplus\ol\twoheadrightarrow \ol$.
\begin{remark}\label{rem}
There exists a unique divisor $H_0$ in $|H|$ that contains $S_1$; $H_0$ is smooth, $H_0\cong \pr_T(\ol(N)\oplus\ol)$, $H_{|H_0}$ is a tautological divisor, and $S_1\sim 2H_{|H_0}$.
\end{remark}
\begin{proof}
  Since $|H|$ contains smooth members, and $S_1$ is a general complete intersection, there exists $H_0\in|H|$ smooth containing $S_1$, and $\ph_{|H_0}\colon H_0\to T$ is a $\pr^1$-bundle. For general $t\in T$, $(H_0)_{|Z_t}$ is a line, and $S_1\in |2H_{|H_0}|$, so that $S_1\cap Z_t$ yields two points which determine uniquely the line  $(H_0)_{|Z_t}$.

  We note that $H_0$ must intersect $D$ in a section $S_0$ of $\ph$, given by a surjection $\lambda\colon \ol\oplus\ol\twoheadrightarrow \ol$, so that $H_0\cong\pr_T(\ma{E})$ where $\ma{E}$ is a rank $2$ vector bundle on $T$. We have commutative diagrams:
  $$\xymatrix{{Z}
    &{H_0}\ar@{_{(}->}[l]
   \\
{D}\ar@{_{(}->}[u]&{S_0}\ar@{_{(}->}[l]\ar@{_{(}->}[u]
}\qquad\qquad
\xymatrix{{\ol(N)\oplus\ol\oplus\ol}
    \ar@{->>}[r]\ar@{->>}[d]^{\pi}&{\ma{E}}\ar@{->>}[d]^{\sigma}\\
{\ol\oplus\ol}\ar@{->>}[r]^{\lambda}&{\ol}
}$$
Since $\lambda$ and $\pi$ have a section, this yields a section of $\sigma$, which implies that $\ma{E}\cong \ol(N)\oplus\ol$ and gives the rest of the statement.
\end{proof}  
\begin{remark}\label{disj}
We have $H_0\cap S_2=H_0\cap S_3=D\cap S_1=\emptyset$; in particular $S_1,S_2,S_3$ are pairwise disjoint.
\end{remark}
\begin{proof}
Every divisor in $|H|$ intersects $D$ in $\{pt\}\times T$, and so it follows that $H_0$ is disjoint from $S_2$ and $S_3$ by generality.
  
Furthermore, the section $H_0 \cap D$ of $H_0$ has normal bundle $\ma{N}_{H_0 \cap D/H_0} \cong \ol_T(-N)$, and is given by a surjection $\ol(N)\oplus \ol\twoheadrightarrow \ol$.

Note that $S_1$ does not intersect any section of $H_0$ with normal bundle $\ol_T(-N)$, and it follows that $S_1$ is disjoint from $D$.
\end{proof}
\begin{remark}\label{doublecover}
  $S_1$ is irreducible and $\ph_{|S_1}\colon S_1\to T$ is finite of degree $2$; moreover $2N$ is linearly equivalent to the branch divisor $\Delta\subset T$,   $h^0(T,2N)>0$, and $H_{|S_1}\sim (\ph_{|S_1})^*(N)\sim R$ where  $R\subset S_1$ is the ramification divisor.
  \end{remark}
\begin{proof}
The restriction $\ph_{|S_1}\colon S_1\to T$ is finite of degree $2$, because if $\dim(S_1\cap Z_t)>0$ for some $t\in T$, we should have $S_1\cap D\neq \emptyset$, contradicting Rem.~\ref{disj}. Let $R\subset S_1$ be the ramification divisor and $\Delta\subset T$ the branch divisor.
By
adjunction
$$-K_{S_1}=(-K_Z-3{H})_{|S_1}=(\ph_{|S_1})^*\ol_T(-K_T-N),$$ 
so by the Hurwitz formula $ (\ph_{|S_1})^*(N)\sim R$ and
$2N\sim\Delta$. In particular $\ph_{|S_1}$ is not \'etale because $N\not\equiv 0$, and we deduce that $S_1$ is irreducible. We also note that
 $h^0(T,2N)>0$ and
 that $S_1$ is Fano if and only if $-K_T-N$ is ample. Finally we have  $D\sim H-\ph^*(N)$ (see Rem.~\ref{relative}) and $D\cap S_1=\emptyset$ by Rem.~\ref{disj}, thus $H_{|S_1}\sim \ph_{|S_1}^*(N)$.
\end{proof}
  Let $h\colon X\to Z$ be the blow-up of $S_1,S_2,S_3$, so that $X$ is a smooth projective variety with $\rho_X=\rho_T+4$, and set $\sigma:=\ph\circ h\colon X\to T$.
\begin{lemma}\label{fanoB}
The following are equivalent:
\begin{enumerate}[$(i)$]
\item
$X$ is Fano;
\item
 $-K_T\pm N$ is ample on $T$.
\end{enumerate}
\end{lemma}
\begin{proof}
Let $E_i\subset X$ be the exceptional divisor over $S_i$, and let $\w{H}_0$ and $\widetilde D$ be the transforms of $H_0$ and $D$. 

We show $(i)\Rightarrow (ii)$. If $X$ is Fano, by restricting $-K_X$ to $E_2$ (or $E_3$) and using Lemma \ref{E} we see that $-K_T-N$ is ample. 
Then restricting $-K_X$ to $\widetilde D$ and using Lemma \ref{E'} we see that $-K_T+N$ is ample.

Let us show the converse  $(ii)\Rightarrow (i)$. 
Again, it follows directly from Lemmas \ref{E} and \ref{E'} that $(-K_X)_{|E_i}$ for $i=2,3$ and $(-K_X)_{|\widetilde D}$ are ample divisors respectively on $E_i$ and $\widetilde D$.

We show that $(-K_X)_{|\w{H}_0}$ is ample on $\w{H}_0$. Since $\w{H}_0=h^*{H_0}-E_1$ by Rem.~\ref{rem} and \ref{disj},
we have $\ol_X(\w{H}_0)_{|\w{H}_0}\cong \ol_Z(H_0)_{|H_0} \otimes \ol_{H_0}(-S_1)\cong \ol_{H_0}(-H_{|H_0})$ 
under the isomorphism $\w{H}_0\cong H_0$. Set $\ph_0:=\ph_{|H_0}\colon H_0\to T$. Using  Rem.~\ref{rem} we get:
\begin{align*}
\ol_X(-K_X)_{|{\widetilde H_0}} &\cong \ol_{\w{H}_0}(-K_{\w{H}_0}) \otimes \ol_X(\w{H}_0)_{|\w{H}_0} \cong \ol_{H_0}(-K_{H_0}-H_{|H_0})\\ &\cong
\ol_{H_0}\bigl(\ph_0^*(-K_T-N) +H_{|H_0}\bigr)
\end{align*}
which is the tautological line bundle of $\pr_T(\ol(-K_T)\oplus\ol(-K_T-N))$, and hence it is ample.

We show that $-K_X\cdot\Gamma>0$ for every irreducible curve $\Gamma\subset X$. If $\Gamma$ is contained in one of the divisors $E_2, E_3, \widetilde D,\w{H}_0$, then 
  $-K_X\cdot\Gamma>0$ by what precedes.

  Now assume that $\Gamma\not\subset E_2, E_3, \w{H}_0, \widetilde D$.
  We have $D\sim H_0-\ph^*(N)$ (see Rem.~\ref{relative}) and
  $-K_Z\sim\ph^*(-K_T-N)+3H_0\sim \ph^*(-K_T+N)+H_0+2D$. Then
    it follows from Rem.~\ref{disj} that:
$$-K_X\sim h^*(-K_Z)-E_1-E_2-E_3 \sim\sigma^*(-K_T+N)+\w{H}_0+2\widetilde D+E_2+E_3.$$
 
If $\Gamma$ is not contracted by $\sigma$, then $\sigma(\Gamma)$ is an irreducible curve in $T$, and  $-K_X\cdot\Gamma=(-K_T+N)\cdot\sigma_*(\Gamma) + (E_2+E_3+\w{H}_0+2\widetilde D)\cdot\Gamma>0$.

If instead $\sigma(\Gamma)$ is a point, then either $\Gamma$ is a fiber of $h$, and so $-K_X\cdot\Gamma=1$, or $h(\Gamma)$ is an irreducible curve in a fiber $Z_t$ of $\varphi$, for some $t\in T$. Assume the latter: then $-K_X\cdot\Gamma=(\w{H}_0+\widetilde D+h^*(D))\cdot\Gamma\geq h^*(D)\cdot \Gamma=D\cdot h_*(\Gamma)>0$,
because $D_{|Z_t}$ is a line in $Z_t\cong \pr^2$.

\medskip

Now we show that $-K_X$ is big. Consider the divisor $A_Z:=H_0+ \varphi^*(-K_T)$, which is a tautological divisor for $\pr_T(\ol(-K_T+N)\oplus\ol(-K_T)\oplus\ol(-K_T))$, and hence it
  is ample on $Z$.
 Therefore there  exist non-negative integers $m,c_i$ such that the divisor
 \begin{align*}
   A_X&:=m h^*(A_Z) - \sum_{i=1}^3c_iE_i\\
&=m\bigl(\w{H}_0+E_1+\sigma^*(-K_T)\bigr)- \sum_{i=1}^3c_iE_i\\
&  =m\sigma^*(-K_T)+m\w{H}_0+(m-c_1)E_1-c_2E_2-c_3E_3
   \end{align*}
is ample on $X$. Finally
\begin{equation*}
  -mK_X\sim m\bigl(\sigma^*(-K_T)+E_1+2\w{H}_0+\widetilde D\bigr)=A_X+m\w{H}_0+m\widetilde D + \sum_{i=1}^3c_iE_i
\end{equation*}
which implies that $-K_X$ is big and then ample, as in the proof of Lemma \ref{fanoA}.
  \end{proof}
The proof of the following lemma is very similar to the one of Lemma \ref{deltaA}, thus we omit it.
\begin{lemma}\label{deltaB}
  Suppose that $X$ is Fano.  Then $\delta_X\geq 3$, and $\delta_X=3$ if and only if $\delta_T\leq 3$.
 \end{lemma}
Prop.~\ref{B} follows from Lemmas \ref{fanoB} and \ref{deltaB}.
\section{Proof of Theorem \ref{main}}\label{proof}
\noindent We follow \cite[proof of Th.~1.1]{delta3_4folds}, where the case $\dim X=4$, $\rho_X=5$, and $T\cong\pr^2$ is considered, generalizing the proof step by step.
We begin by recalling a result on the structure of Fano varieties $X$ with $\delta_X=3$.
\begin{thm}\label{factorization} Set $n:=\dim X$.  There exists a diagram:
$$\xymatrix{
X\ar@/^{1pc}/[rrr]^{\sigma}\ar[r]_f &X_{2}\ar[r]_{\psi}&Y\ar[r]_{\xi}& T}$$
with the following properties:
\begin{enumerate}[$(a)$]
\item \label{sp} all varieties are smooth and projective;
\item \label{F} $Y$ is Fano of dimension $n-1$ and $T$ is Fano of dimension $n-2$;
\item \label{P1} $\psi$ and $\xi$ are smooth morphisms with fiber $\pr^1$;
\item \label{flat} $\sigma\colon X\to T$ is a flat contraction of relative dimension $2$;
\item \label{conic} $\psi\circ f\colon X\to Y$ is a conic bundle;
\item \label{Bi} $f$ is the blow-up of two disjoint smooth, irreducible subvarieties $B_1,B_2\subset X_2$ of codimension $2$;
\item \label{Ai} for $i=1,2$ set $A_i:=\psi(B_i)\subset Y$
and $W_i:=\psi^{-1}(A_i)\subset X_2$; $A_1$ and $A_2$ are disjoint smooth divisors in $Y$, and  $B_i$ is a section of $\psi_{|W_i}\colon W_i\to A_i$, for $i=1,2$;
\item \label{xi} $\xi$ is finite on $A_i$ for $i=1,2$.
\end{enumerate}
\end{thm}
\begin{proof}
See \cite{codim}, Prop.~3.3.1 and its proof, in particular paragraphs 3.3.15-3.3.17; see also \cite[Prop.~3.4 and 3.5]{eleonora} for $(\ref{Ai})$.
\end{proof}
We keep the same notation as in Th.~\ref{factorization}.
\begin{step}\label{sectionA}
 By Lemma \ref{levico}, up to switching $A_1$ and $A_2$, we can assume that $A_2$ is a section of $\xi\colon Y\to T$. 
\end{step}
\begin{step}\label{Z}
We can assume that there exists a commutative diagram:
$$\xymatrix{{X_2}\ar[r]^\psi\ar[d]_g&Y\ar[d]^{\xi}\\
Z\ar[r]^{\ph} &{T} 
} $$
where  $\varphi$ is a $\pr^2$-bundle, $\psi$ and $\xi$ are $\pr^1$-bundles, $g$ is the blow-up along a section $S_3\subset Z$ of $\ph$, $E:=\Exc(g)$ is a section of $\psi$, and $E\cap (B_1\cup B_2)=\emptyset$.
\end{step}
\begin{proof}
Consider the natural factorization of $f$ as a sequence of two blow-ups (see Th.~\ref{factorization}($\ref{Bi}$)):
$$X\stackrel{f_1}{\la} X_{1}\stackrel{f_2}{\la} X_2$$
where $f_2$ is the blow-up of $B_2$ and $f_1$ is the blow-up of the transform  of $B_1$.

Let us consider the morphism $\zeta:=\xi\circ \psi\circ f_2\colon X_1\to T$. 
Since both $\psi$ and $\xi$ are smooth 
by Th.~\ref{factorization}($\ref{P1}$), 
the composition $\xi\circ \psi \colon X_2\to T$ is smooth. Moreover, since $A_2\subset Y$ is a section of $\xi$ by Step \ref{sectionA}, and the center $B_2$ of the blow-up $f_2\colon X_1\to X_2$ is a section over $A_2$ 
(see Th.~\ref{factorization}($\ref{Ai}$)), 
we conclude that $B_2$ is a section of  $\xi\circ \psi \colon X_2\to T$. This implies that $\zeta\colon X_1\to T$
is a smooth contraction of relative Picard number $3$.

As in \cite[proof of Step 2.3]{delta3_4folds} one shows that  $-K_{X_1}$ is 
$\zeta$-ample;
every fiber of $\zeta$ is isomorphic to the blow-up $F$ of $\pr^2$ in two points. 
Moreover, again as in \cite[proof of Step 2.3]{delta3_4folds} we see that
every contraction of the fiber $F$ extends to a global contraction of $X_1$ over $T$.
Therefore the sequence of elementary contractions:
$$\xymatrix{F\ar[r]&{\mathbb{F}_1}\ar[d]\ar[r]&{\pr^1}\ar[d]\\
&{\pr^2}\ar[r]&
\{pt\}}$$
yields a corresponding factorization of $\zeta$:
$$\xymatrix{{X_1}\ar[r]^{f_2'}&{X_2'}\ar[d]_g\ar[r]^{\psi'}&{Y'}\ar[d]^{\xi'}\\
&{Z}\ar[r]^{\ph}&{T}}$$
We have:
\begin{enumerate}[$\bullet$]
\item
$\xi'\colon Y'\to T$ and $\psi'\colon X_2' \to Y'$ are smooth with fiber $\pr^1$, and $\ph\colon Z\to T$ is smooth with fiber   $\pr^2$;
\item $g$ is the blow-up of a smooth irreducible subvariety $S_3\subset Z$ of codimension $2$, which is a section of $\ph$;
\item
 $f_2'$ is the blow-up of a smooth irreducible subvariety $B_2'\subset X_2'$
of codimension $2$, which is a section of $\ph\circ g\colon X_2'\to T$, and is disjoint from $E:=\Exc(g)$;
\item $A_2':=\psi'(B_2')$ is a section of $\xi'$;
\item
 $E$ is a section of $\psi'\colon X_2'\to Y'$;
\item since $\ph$, $\xi'$, and $\psi'$ all have a section, they are projectivizations of vectors bundles.
\end{enumerate}

Notice that $\zeta$ is finite on $f_1(\Exc(f_1))$, because $f_2$ is an isomorphism on it, $\psi$ is finite on $B_1$, and $\xi$ is finite on $A_1$
(see  Th.~\ref{factorization}($\ref{Bi}$),($\ref{Ai}$),($\ref{xi}$)).
In particular $f_1(\Exc(f_1))$ cannot contain any $(-1)$-curve in the fiber $F$, therefore it cannot  meet any such curve, otherwise $X$ would not be Fano. Hence $B_1':=f_2'(f_1(\Exc(f_1)))$ is disjoint from $B_2'$ and from $E$.

We show that properties  of Th.~\ref{factorization} still hold, so that we can replace the original factorization of $\zeta$ with the new one.
Properties ($\ref{sp}$),($\ref{P1}$),($\ref{Bi}$) are clear. Since  $\zeta$ is finite on $f_1(\Exc(f_1))$, $\xi'$ is finite on $A_1':=\psi'(B_1')$, and we have ($\ref{xi}$).

 The composition $\psi'\circ f_2'\circ f_1\colon X\to Y'$ is a fiber type contraction with one-dimensional fibers, hence a conic bundle, which yields 
 ($\ref{conic}$) and ($\ref{flat}$). Moreover this conic bundle has only reduced fibers; this implies that $Y'$ is Fano by \cite[Prop.~4.3]{wisn}, and we have ($\ref{F}$). Finally ($\ref{Ai}$) follows from \cite[Prop.~3.4 and 3.5]{eleonora}.
\end{proof}
Set $S_i:=g(B_i)\subset Z$ for $i=1,2$.
Then $S_1$, $S_2$, and $S_3$ are pairwise disjoint smooth irreducible subvarieties of codimension $2$, and 
$X$ is the blow-up of $Z$ along $S_1\cup S_2\cup S_3$. We set $Z_t:=\ph^{-1}(t)$ for every $t\in T$.  Moreover we denote by $d$ the degree of the finite morphism $\xi_{|A_1}\colon A_1\to T$
(see Th.~\ref{factorization}$(\ref{xi})$).   
\begin{step}\label{d}
 $S_2$ is a section of $\ph$, and $\ph_{|S_1}\colon S_1\to T$ is finite of degree $d$.
\end{step}
\begin{proof}
For $i=1,2$, since $B_i$ is a section over $A_i$
by Th.~\ref{factorization}$(\ref{Ai})$, 
the degree of $S_i$ over $T$ is equal to the degree of $A_i$ over $T$; in particular $S_2$ is a section of $\ph$
 by Step \ref{sectionA}.
\end{proof}
\begin{step}\label{points}
  For every $t\in T$ and for every line $\ell\subset Z_t\cong\pr^2$, the scheme-theoretic intersection $\ell\cap (S_1\cup S_2\cup S_3)$ has length $\leq 2$.
  
  In particular, the points  $(S_1\cup S_2\cup S_3)\cap Z_t$ (with the reduced structure) are in general linear position in $Z_t$.
  \end{step}
\begin{proof}
 Let $a$ be the length of scheme-theoretic intersection $\ell\cap(S_1\cup S_2\cup S_3)$, and let $\w{\ell}\subset X$ be the transform of $\ell$. Then $1\leq -K_X\cdot \w{\ell}=3-a$, thus $a\leq 2$.
\end{proof}
\begin{step}\label{U}
  If $d=1$, then $X\to Z\stackrel{\ph}{\to} T$ is as in Construction A (Prop.~\ref{A}).
\end{step}
\begin{proof}
If $d=1$, then
$\ph\colon Z\to T$ has three pairwise disjoint sections $S_i$, which 
are fiberwise in general linear position, 
by Steps \ref{Z}, \ref{d}, and \ref{points}. 
This implies that $Z\cong \pr_{T}(\ol(D_1)\oplus \ol(D_2)\oplus \ol(D_3))$ in such a way that the three sections $S_i$ correspond to the projections onto the summands $\ol(D_i)$. Moreover $-K_T+D_i-D_j$ is ample for every $i,j\in\{1,2,3\}$ by Lemma \ref{fanoA}, and $\delta_T\leq 3$ by Lemma \ref{deltaA}.
\end{proof}
From now on we assume  that  $d\geq 2$.

For $q_1,q_2\in Z_t$ distinct points, we denote by $\overline{q_1q_2}\subset Z_t$ the line spanned by $q_1$ and $q_2$.
\begin{step}\label{H}
  Let $H_0\subset Z$ be the relative secant variety of $S_1$ in $Z$, namely the closure in $Z$ of the set:
  $$\bigcup_{q_1,q_2\in S_1\cap Z_t,q_1\neq q_2,t\in T}\overline{q_1q_2}.$$
 
Then
$H_0$ has pure dimension $n-1$, and 
for every $t\in T$, $H_0\cap Z_t$ is a union of lines $\ell$ such that  the scheme-theoretic intersection $\ell\cap S_1$ has length $\geq 2$. 
\end{step}
\begin{proof}
 For $t$ general we have $|S_1\cap Z_t|=d\geq 2$, so that  $H_0$ is non-empty.
    
We first consider the fiber product:
$$\xymatrix{{S_1\times_{T} S_1}\ar[d]\ar[r]&{S_1}\ar[d]\\
{S_1}\ar[r]& T
}$$
Since the morphism $S_1\to T$ is finite between smooth varieties, it is flat. Therefore
also $S_1\times_{T} S_1\to T$ is finite and flat, so that 
$S_1\times_{T} S_1$ has pure dimension $n-2$ and every irreducible component dominates $T$.

We also note that the morphism $S_1\times_{T} S_1\to S_1$ has a natural section, whose image is the diagonal $\Delta$, which is an irreducible component of the fiber product. We denote by $(S_1\times_{T} S_1)_0$ the union of the remaining irreducible components of the fiber product, so that $(S_1\times_{T} S_1)_0\smallsetminus\Delta$ is a dense open subset in $(S_1\times_{T} S_1)_0$.

Now we consider the dual projective bundle $Z^*\to T$. We have a morphism over $T$:
$$\alpha\colon (S_1\times_{T} S_1)_0\smallsetminus\Delta\la Z^*$$
given by $\alpha(q_1,q_2)=[\overline{q_1q_2}]$. If $[\ell]\in \text{Im}(\alpha)$, then $\ell\cap S_1$ is finite, hence $\alpha^{-1}([\ell])$ is finite. Moreover
$$\text{Im}(\alpha)=\{[\overline{q_1q_2}]\in Z_t^*\,|\,q_1,q_2\in S_1\cap Z_t,q_1\neq q_2,t\in T\}.$$

Let $\ma{S}\subset Z^*$ be the closure of $\text{Im}(\alpha)$.
If $K$ is an irreducible component of $(S_1\times_{T} S_1)_0$, 
set $\ma{S}_K:=\overline{\alpha(K\smallsetminus\Delta)}$.
Then 
$\alpha\colon K\smallsetminus\Delta\to\ma{S}_K$ is a dominant morphism between irreducible varieties, hence its image $\alpha(K\smallsetminus\Delta)$ contains an open subset of  $\ma{S}_K$; moreover
$\dim \ma{S}_K=\dim K=n-2$ and
$\ma{S}_K$ dominates $T$, because $K$ does.

We note that $\ma{S}$ is the union of $\ma{S}_K$ when $K$ varies among the finitely many irreducible components  of $(S_1\times_{T} S_1)_0$, therefore all the irreducible components of 
$\ma{S}$ are of type $\ma{S}_K$.

We conclude that $\ma{S}$ has pure dimension $n-2$ and every irreducible component dominates $T$. Moreovery every  irreducible component of $\ma{S}$ contains an open subset contained in $\text{Im}(\alpha)$.
We also note that, for every $[\ell]\in\ma{S}$, the scheme-theoretical intersection $\ell\cap S_1$ has length $\geq 2$.

Now we consider the universal family over $T$:
$$ \xymatrix{&{\ma{U}}\ar[dl]\ar[dr]^{\pi}&\\
Z^*&&Z
}$$
where  $\ma{U}=\{([\ell],p)\in Z^*\times_{T} Z\,|\,p\in\ell\}$.

The inverse image $\wi{\ma{S}}$ of $\ma{S}$ in $\ma{U}$  has pure dimension $n-1$; moreover every irreducible component of $\wi{\ma{S}}$ is the inverse image of an irreducible component of $\ma{S}$,  and  contains an open subset contained in the 
 inverse image $\wi{\ma{S}}_0$ of $\text{Im}(\alpha)$.
In particular $\wi{\ma{S}}_0$ is dense in $\wi{\ma{S}}$.

Finally we consider the images $\pi(\wi{\ma{S}}_0)$ and $\pi(\wi{\ma{S}})$ in $Z$; note that $\pi(\wi{\ma{S}}_0)$ is dense in $\pi(\wi{\ma{S}})$. By construction we have
$$\pi(\wi{\ma{S}}_0)= \bigcup_{q_1,q_2\in S_1\cap Z_t,q_1\neq q_2,t\in T}\overline{q_1q_2}\subset Z,$$
so that $H_0=\overline{\pi(\wi{\ma{S}}_0)}=\pi(\wi{\ma{S}})$.

Let $W$ be an irreducible component of $\wi{\ma{S}}$, and let $W_0$ be an open subset of $W$ contained in $\wi{\ma{S}}_0$. Then $\overline{\pi(W_0)}=\overline{\pi(W)}$, so that $\pi_0\colon W_0\to \overline{\pi(W)}$ is a dominant morphism between irreducible varieties, and its image contains an open subset $U$ of $\overline{\pi(W)}$. It is clear that  $\pi_0^{-1}(p)$ is finite for every $p\in U$, so that
$\dim \overline{\pi(W)}=\dim W=n-1$.
This shows that $H_0$ has pure dimension $n-1$. 

Moreover, since $H_0$ is the locus of lines parametrized by $\ma{S}$, we also have that for every $t\in T$, $H_0\cap Z_t$ is a union of finitely many lines $\ell$, and  the scheme-theoretic intersection $\ell\cap S_1$ has length $\geq 2$.
\end{proof}
\begin{step}\label{disjoint}
We have $H_0\cap(S_2\cup S_3)=\emptyset$.
\end{step}
\begin{proof}
Let  $\ell$ be a line in $H_0\cap Z_t$ for some $t\in T$; then
$\ell\cap(S_2\cup S_3)=\emptyset$ by Steps \ref{points} and \ref{H}.
Therefore  $H_0\cap(S_2\cup S_3)=\emptyset$.
\end{proof}
Recall that $W_2=\psi^{-1}(A_2)\subset X_2$ and that $E=\Exc(g)\subset X_2$
 (see Th.~\ref{factorization}($\ref{Ai}$) and Step \ref{Z}).
\begin{step}\label{D}
Set $D:=g(W_2)\subset Z$. Then $W_2\cong D\cong\pr^1\times T$, and $D\cap Z_t$ is a line in $Z_t$ for every $t\in T$. Moreover $D$ contains $S_2$ and $S_3$ (as $\{pt\}\times T$),  while $D\cap S_1=\emptyset$.
\end{step}
\begin{proof}
We have a commutative diagram:
$$\xymatrix{{W_2}\ar[d]_{g_{|W_2}}\ar[r]^{\ \psi_{|W_2}}&
{A_2}\ar[d]^{\xi_{|A_2}}\\
D\ar[r]_{\ph_{|D}}&{T}
}$$
where the vertical  maps are isomorphisms 
by Steps \ref{sectionA} and \ref{Z}, 
and the  horizontal  maps are $\pr^1$-bundles.

We also have $S_3=g(E\cap W_2)\subset D$; moreover $B_2\subset W_2$
(see Th.~\ref{factorization}($\ref{Ai}$)), 
hence $S_2=g(B_2)\subset D$.
By Steps \ref{Z} and \ref{d}, $S_2$ and $S_3$ are disjoint sections of the $\pr^1$-bundle $\ph_{|D}\colon D\to T$. Thus we can assume
that $D\cong\pr_{T}(\ol\oplus M)$ for some $M\in\Pic(T)$, such that the section $S_2$ corresponds to the projection $\ol\oplus M\twoheadrightarrow \ol$, so that $\ma{N}_{S_2/D}^{\vee}\cong M$.

Now we have $H_0\cap D\neq\emptyset$, because both  contain a line in $Z_t$, so that $H_0\cap D$ yields a non-zero effective divisor in $D$. On the other hand, this divisor is disjoint from both sections $S_2$ and $S_3$ by Step \ref{disjoint}.
Write $\ol_D(H_0\cap D)\cong\ol_D(mS_2+(\ph_{|D})^*G)$, with $m\in\Z$ and $G$ divisor on $T$.

We have $\ol_{S_3}\cong\ol_D(H_0\cap D)_{|S_3}\cong\ol_D((\ph_{|D})^*G)_{|S_3}$, which gives $G\sim 0$ and $\ol_D(H_0\cap D)\cong\ol_D(mS_2)$. Note that $m\neq 0$ because $H_0\cap D$ is non-zero.

Then $\ol_{S_2}\cong\ol_D(H_0\cap D)_{|S_2}\cong\ol_D(mS_2)_{|S_2}\cong M^{\otimes (-m)}$, thus $M^{\otimes m}\cong\ol_T$. Since $T$ is a Fano variety, $\Pic(T)$ is finitely generated and free, so that $M\cong\ol_T$ and 
 $D\cong\pr^1\times T$.

Finally, we note that for every $t\in T$, $D\cap Z_t$ is the line spanned by the points $S_2\cap Z_t$ and $S_3\cap Z_t$, so that $D\cap S_1=\emptyset$ by Step \ref{points}.
\end{proof}
Recall from Steps \ref{sectionA} and \ref{Z} that $\xi\colon Y\to T$ is a $\pr^1$-bundle and $A_2\subset Y$ is a section. Let $\ma{E}$ be a rank $2$ vector bundle on $T$ such that $Y=\pr_T(\ma{E})$; up to tensoring $\ma{E}$ with a line bundle, we can assume that the section $A_2$  corresponds to a surjection $\ma{E}\twoheadrightarrow\ol_T$, and we denote by $L\in\Pic(Y)$ the tautological line bundle, so that $L_{|A_2}\cong\ol_{A_2}$. Moreover let $N$ be a divisor on $T$ such that $\ol_T(N)\cong\ma{N}_{A_2/Y}^{\vee}$ under the isomorphism $T\cong A_2$, so that
 we have an exact sequence on $T$:
\begin{equation}\label{sequenceF}
0\la \ol_T(N)\la\ma{E}\la\ol_T\la 0.
\end{equation}
\begin{step}\label{vanishing}
We have $\ol_Y(A_1)\cong L^{\otimes d}$, $\ol_Y(A_2+\xi^*(N))\cong L$,  and $H^1(Y,L)=0$.
\end{step}
\begin{proof}
  Both $L$ and $\ol_Y(A_1)$ belong to the kernel of the restriction $r\colon \Nu(Y)\to\Nu(A_2)$; moreover $r$ is surjective, because for $G\in\Pic(A_2)$ the line bundle $\xi^*(((\xi_{|A_2})^{-1})^*(G))$ restricts to $G$, and $\rho_{A_2}=\rho_T=\rho_Y-1$, so that $\dim\ker r=1$ and the classes of $L$ and $A_1$ must be proportional in $\Nu(Y)$; intersecting with a fiber of $\xi$ we deduce that  $\ol_Y(A_1)\cong L^{\otimes d}$.

Similarly, we have  $\ol_Y(A_2+\xi^*(N))_{|A_2}\cong\ol_{A_2}$, and again intersecting with a fiber we get $\ol_Y(A_2+\xi^*(N))\cong L$.

  Finally we have 
$$-K_Y+L\equiv -K_Y+\frac{1}{d} A_1$$
where $-K_Y$ is ample, $1/d<1$, and $A_1$ is a smooth divisor, therefore
$H^1(Y,L)=0$
 by Kawamata-Viehweg vanishing (see \cite[Th.~9.1.18]{lazII}).
\end{proof}
\begin{step}\label{normal}
We have $X_2\cong \pr_Y(\ol\oplus{L})$,
 and $E$ corresponds (as a section of $\psi$) to the projection $\ol\oplus{L}\twoheadrightarrow\ol$. Moreover  $\ma{N}^{\vee}_{E/X_2}\cong  L$ and
 ${\ma{N}}_{S_3/Z}^{\vee}\cong \ma{E}$. 
\end{step}
\begin{proof}
By Step \ref{Z} we have a commutative diagram:
$$\xymatrix{{E}\ar[d]\ar[r]^{\psi_{|E}}& Y\ar[d]^{\xi}\\
{S_3}\ar[r]^{\ph_{|S_3}}&{T}
}$$
where the horizontal maps are isomorphisms, and the vertical maps are $\pr^1$-bundles.
Since $g\colon X_2\to Z$ is the blow-up of $S_3$, we get:
$$\mathbb{P}_{S_3}({\ma{N}}_{S_3/Z}^{\vee})\cong E\cong Y\cong \pr_{T}(\ma{E}),$$ 
therefore ${\ma{N}}_{S_3/Z}^{\vee}\cong\ma{E}\otimes M$ with $M\in\Pic(T)$. Moreover $\ma{N}^{\vee}_{E/X_2}$ is the tautological line bundle of $\mathbb{P}_{T}(\ma{E}\otimes M)$, which gives $\ma{N}^{\vee}_{E/X_2}\cong L\otimes \xi^*(M)$.

Recall from Step \ref{Z} that $E\subset X_2$ is a section of the
$\pr^1$-bundle $\psi\colon X_2\to Y$.
Let $\ma{F}$ be a rank $2$ vector bundle on $Y$ such that $X_2=\pr_Y(\ma{F})$;
up to tensoring $\ma{F}$ with a line bundle, we can assume that
the section $E$ corresponds to a surjection $\sigma\colon \mathcal{F}\twoheadrightarrow \mathcal{O}_Y$, so that $\ker\sigma\cong \ma{N}_{E/X_2}^\vee\cong L\otimes\xi^*(M)$.  We  obtain the following exact sequence over $Y$:
\begin{equation}\label{sequence}
0\longrightarrow \ker\sigma\longrightarrow \mathcal{F}\longrightarrow \mathcal{O}_Y\longrightarrow 0.\end{equation}

Now let us consider $A_2\subset Y$; 
we have $\ker\sigma_{|A_2}\cong 
L_{|A_2}\otimes\xi^*(M)_{|A_2}\cong M$, so by restricting to $A_2$ the above exact sequence  we get:
$$0\longrightarrow M\longrightarrow \mathcal{F}_{|A_2}\longrightarrow \mathcal{O}\longrightarrow 0.$$
On the other hand $\pr_{A_2}(\mathcal{F}_{|A_2})=W_2\cong \pr^1\times T$ by Step \ref{D}, and we deduce that $M=\ol_T$ and $\ker\sigma\cong L$.

Now we have $\text{Ext}^1(\mathcal{O}_Y,\ker\sigma)\cong H^1(Y,L)=0$ by Step \ref{vanishing}, 
hence  the  sequence \eqref{sequence} splits, so that $\mathcal{F}\cong\mathcal{O}_Y\oplus L$.
\end{proof}
\begin{step}\label{N}
  We have $-K_T+N$ ample on $T$, $H^1(T,N)=0$,  $H^0(T,d(d-1)N)\neq 0$, $N\not\equiv 0$,
  $\ma{E}\cong \ol_T(N)\oplus \ol_T$, and $Y\cong\pr_T(\ol(N)\oplus\ol)$. \end{step}
\begin{proof}
Set $\xi_1:=\xi_{|A_1}\colon A_1\to T$, so that $\xi_1$ is finite of degree $d$.
By Step \ref{vanishing} we have  $L\cong\ol_Y(A_2+\xi^*(N))$ and
 $\ol_Y(A_1)\cong L^{\otimes d}$, so that $L_{|A_1}\cong\ol_{A_1}(\xi_1^*(N))$
 and $\ol_Y(A_1)_{|A_1}\cong\ol_{A_1}(\xi_1^*(dN))$. We have
 $\det(\ma{E})\cong\ol_T(N)$ by
 \eqref{sequenceF}, therefore
$\ol_Y(K_Y)\cong \ol_Y(\xi^*(K_T+N))\otimes L^{\otimes(-2)}$ and
$$-K_{Y|A_1}\sim\xi_1^*(-K_T+N).$$
Since $Y$ is Fano and $\xi_1$ is finite, we deduce that $-K_T+N$ is ample on $T$, so that $H^1(T,N)=H^1(T,K_T-K_T+N)=0$ by Kodaira vanishing.
Moreover we have $\Ext^1(\ol_T,\ol_T(N))=0$, so that the sequence \eqref{sequenceF} splits.

Finally:
$$K_{A_1}\sim (K_Y+A_1)_{|A_1}\sim\xi_1^*\bigl(K_T+(d-1)N\bigr).$$
This shows that the ramification divisor $R$ of $\xi_1$ is linearly equivalent to
$(d-1)\xi_1^*(N)$; by taking the pushforward we get $(\xi_1)_*(R)\sim d(d-1)N$.
On the other hand $R$ is effective and non-zero (because  $T$ is simply connected), therefore $H^0(T,d(d-1)N)\neq 0$ and $N\not\equiv 0$.
\end{proof}
Recall from Step \ref{D} that $D\cong\pr^1\times T$ and $S_2,S_3\cong\{\text{pt}\}\times T$.
\begin{step} \label{bundleZ}
  We have $Z\cong \pr_{T}(\ol(N)\oplus \ol\oplus \ol)$ with the two sections $S_2,S_3$ corresponding to the projections onto the trivial summands; moreover
 $(\ma{N}_{D/Z})_{|\{pt\}\times T}\cong \ol_T(-N)$.
\end{step}
\begin{proof}
  We know by Steps \ref{Z}, \ref{normal}, and \ref{N} that $S_3$ is a section of $\ph\colon Z\to T$ with conormal bundle $\ma{E}\cong\ol_T(N)\oplus \ol_T$. 
  As in the proof of Step \ref{normal}, using that $H^{1}(T,N)=H^1(T,\ol_T)=0$ by Step \ref{N},  one shows that $Z\cong\pr_{T}(\mathcal{O}(N)\oplus \mathcal{O}\oplus \ol)$, with the section $S_3$ corresponding to the projection $p_3$ onto the last summand.
  
  The inclusions $S_2,S_3\hookrightarrow D\hookrightarrow Z$ yield a diagram: $$\xymatrix{{\ol_T(N)\oplus\ol_T\oplus\ol_T}\ar[r]^{\quad\quad\alpha}&{\ol_T\oplus\ol_T}
    \ar[r]^{\quad\pi_1}\ar[d]^{\pi_2}&{\ol_T}\\
&{\ol_T}&
}$$
where $\pi_i$ are the projections, $\pi_1$ corresponds to $S_2\subset D$,
$\pi_2$ corresponds to $S_3\subset D$,
and $\pi_2\circ\alpha=p_3$.

By Step \ref{N} we have $H^0(T,-N)=0$, thus $\Hom(\ol_T(N),\ol_T)=0$ and $\alpha$ factors through the projection $p_{23}\colon \ol_T(N)\oplus\ol_T\oplus\ol_T\to \ol_T\oplus\ol_T$. Then up to changing the isomorphism $Z\cong \pr_{T}(\ol(N)\oplus \ol\oplus \ol)$ we can assume that
$\alpha=p_{23}$ and $S_2\subset Z$ corresponds to the projection $p_2$ onto the second summand.
\end{proof}
Since $N\not\equiv 0$ by Step \ref{N},
as in the proof of Lemma \ref{levico} we see that
there exists 
 a curve  $C\subset T$ which is a complete intersection  of very ample divisors (general in their linear system) and such that $N\cdot C\neq 0$.
 Set $S:=\xi^{-1}(C)\subset Y$ and  $C_1:=A_{1|S}$.
\begin{step}\label{curve}
 We have
$\ol_S(C_1)=(L_{|S})^{\otimes d}$, $\deg(L_{|C_1})>0$, and 
$L_{|S}$ is nef and big in $S$.
\end{step}
\begin{proof}
 Recall that $\ol_Y(A_1)\cong L^{\otimes d}\cong  \ol_Y(d(A_2+\xi^*(N)))$ by Step \ref{vanishing}, in particular $\ol_S(C_1)=(L_{|S})^{\otimes d}$. 
 We have
 $$(C_1)^2=A_1\cdot C_1=d\bigl(A_2+\xi^*(N)\bigr)\cdot C_1
=d\xi^*(N)\cdot C_1=d^2N\cdot C\neq 0.$$
If $(C_1)^2<0$, $C_1$ should be the negative section of the ruled surface $S$, which is impossible because $\xi_{|C_1}$ has degree $d\geq 2$; therefore $(C_1)^2>0$,
 $C_1$ is nef and big in $S$, and  $L_{|S}$ too. Moreover
$\deg(L_{|C_1})=C_1\cdot L_{|S}>0$.
\end{proof}
Set
$\overline{X}_2:=\psi^{-1}(S)\subset X_2$, $\overline{Z}:=\ph^{-1}(C)\subset Z$, 
and let $\overline{\xi}$, $\overline{\psi}$, $\overline{\ph}$, $\overline{g}$ be the respective restricted morphisms, 
so that we have a diagram: $$\xymatrix{{\overline{X}_2}\ar[r]^{\overline{\psi}}\ar[d]_{\overline{g}}&S\ar[d]^{\overline{\xi}}\\
\overline{Z}\ar[r]^{\overline{\ph}} &{C}
} $$
where all varieties are smooth and projective, and $\dim \overline{X}_2=\dim \overline{Z}=3$.
We also set $\overline{W}_1:=(\overline{\psi})^{-1}(C_1)=W_1\cap \overline{X}_2\subset \overline{X}_2$ and $\overline{B}_1:=B_1\cap \overline{X}_2$, so that $\overline{B}_1$ is a section of $\overline{\psi}_{|\overline{W}_1}\colon \overline{W}_1\to C_1$ (see Th.~\ref{factorization}($\ref{Ai}$)).
\begin{step}\label{sectionK}
There exists a section $K$ of $\overline{\psi}\colon \overline{X}_2\to S$ containing $\overline{B}_1$ and disjoint from $E$.
\end{step}
\begin{proof}
By Step \ref{normal}
we have   $\overline{W}_1=\pr_{C_1}(\ol_{C_1}\oplus L_{|C_1})$. Moreover by Steps \ref{Z} and \ref{normal} we deduce that
$\overline{W}_1\cap E$ is a section of  $\overline{\psi}_{|\overline{W}_1}$, disjoint from $\overline{B}_1$, and 
corresponding to the projection
$\ol_{C_1}\oplus L_{|C_1}\to\ol_{C_1}$.
Then it is not difficult to see that $\overline{B}_1$ corresponds, as a section, to a surjection  $\tau\colon\ol_{C_1}\oplus L_{|C_1}\twoheadrightarrow L_{|C_1}$.

Let us consider the restriction $r\colon \text{Hom}(\ol_S\oplus L_{|S},L_{|S})\to \text{Hom}(\ol_{C_1}\oplus L_{|C_1},L_{|C_1})$.
We have
$$
\text{Hom}(\ol_{C_1}\oplus L_{|C_1},L_{|C_1})\cong \text{Hom}(L_{|C_1}^{\vee}\oplus \ol_{C_1},\ol_{C_1})\cong 
 H^0({C_1}, L_{|C_1})\oplus H^0({C_1},\ol_{C_1}),
$$
and similarly $\text{Hom}(\ol_{S}\oplus L_{|S},L_{|S})\cong H^0(S, L_{|S})\oplus H^0(S,\ol_{S})$. Since the restriction $H^0(S,\ol_S)\to H^0(C_1,\ol_{C_1})$ is an isomorphism,  $r$ is surjective if the restriction $H^0(S, L_{|S})\to H^0(C_1, L_{|C_1})$ is. 

We have an exact sequence of sheaves on $S$:
$$0\la L_{|S}\otimes\ol_S(-C_1)\la L_{|S}\la L_{|C_1}\la 0.$$
Using Step \ref{curve}, Serre duality and Kawamata-Viehweg vanishing:
$$H^1(S, L_{|S}\otimes\ol_S(-C_1))=H^1(S,L_{|S}^{\otimes (1-d)})=H^1(S,K_S\otimes L_{|S}^{\otimes (d-1)})=0$$
because $d\geq 2$.

We conclude that $r$ is surjective, so that $\tau$ extends to a morphism $\tau_S\colon \ol_S\oplus L_{|S}\to L_{|S}$.

We show that $\tau_S$ is surjective. Under the isomorphism
$\text{Hom}(\ol_{S}\oplus L_{|S},L_{|S})\cong \text{Hom}(\ol_{S},L_{|S})\oplus\text{Hom}(L_{|S},L_{|S})\cong \text{Hom}(\ol_{S},L_{|S})\oplus \C$,
$\tau_S$ corresponds to $(\alpha,\lambda)$.

If $\lambda=0$,   then  $\tau_S$ factors through the projection $\ol_S\oplus L_{|S}\twoheadrightarrow \ol_S$, and the same happens by restricting to $C_1$. This is impossible, because $\tau\colon \ol_{C_1}\oplus L_{|C_1}\to L_{|C_1}$ is surjective, and $\deg(L_{|C_1})>0$ by Step \ref{curve}.

We conclude that $\lambda\neq 0$ and $\tau_S$ is surjective, so
it  yields a section $K\subset \overline{X}_2$ extending $\overline{B}_1$.

We show that $K\cap E=\emptyset$. Let us consider the projection $\ol_S\oplus L_{|S}\twoheadrightarrow L_{|S}$ and the corresponding section $\widetilde{K}\subset \overline{X}_2$. Since $E\cap \overline{X}_2$ is a section corresponding to the projection onto the other summand, we have $\w{K}\cap E=\emptyset$. On the other hand, it is easy to check that $K\sim \w{K}$ in $\overline{X}_2$, hence for every curve $C\subset E\cap\overline{X}_2$ we have $K\cdot C=0$. Since $K\neq E\cap \overline{X}_2$, this implies that $K\cap E=\emptyset$.
\end{proof}
\begin{step}\label{2}
  We have $d=2$. 
\end{step}
\begin{proof}
Recall that $K\cap E=\emptyset$ and  $K\supset \overline{B}_1$ by Step \ref{sectionK}.
   Consider $\overline{g}(K)\subset \overline{Z}$, so that $K\cong \overline{g}(K)$  and 
$\overline{g}(K)\supset S_1\cap \overline{Z}$.
  If $t\in C$ is general, 
then $g^{-1}(Z_t)\cong\mathbb{F}_1$, and $K\cap g^{-1}(Z_t)$ is a section of $\mathbb{F}_1\to\pr^1$, disjoint from the $(-1)$-curve $E\cap g^{-1}(Z_t)$.
Thus $\overline{g}(K)\cap Z_t$ is a line in $Z_t\cong\pr^2$, and this line contains the $d$ points $S_1\cap Z_t$. Since these points  are in general linear position (see Step \ref{points}), we conclude that $d=2$.
\end{proof}
\begin{step}\label{H0}
  The divisor $H_0$ is a tautological divisor for $Z=\pr_T(\ol(N)\oplus\ol\oplus\ol)$.
\end{step}
\begin{proof}
  Since $d=2$ by Step \ref{2}, for a general fiber $Z_t$ of $\ph$ the restriction $(H_0)_{|Z_t}$ is a line (see Step \ref{H}), therefore if $H$ is a tautological divisor, we have $H_0\sim H+\ph^*(G)$, $G$ a divisor on $T$.
 Restricting to $S_3$ we have $H_0\cap S_3=\emptyset$ by Step \ref{disjoint} and $\ol_Z(H)_{|S_3}\cong \ol_{S_3}$ by Step \ref{bundleZ}, thus $G\sim 0$ and $H_0\sim H$.
\end{proof}
\begin{step}\label{vanishing2}
We have $H^1(Z,H_0)=0$.
 \end{step}
 \begin{proof}
 By Step \ref{H0} the divisor $A:=\ph^*(-K_T)+H_0$ is the tautological divisor for $\pr_T(\ol(-K_T+N)\oplus\ol(-K_T)\oplus \ol(-K_T))$, hence it is ample by Th.~\ref{factorization}($\ref{F}$) and Step \ref{N}.

   Consider the divisors $H_0',H_0''\subset Z$ corresponding to the two projections $\ol_T(N)\oplus\ol_T\oplus\ol_T\twoheadrightarrow\ol_T(N)\oplus\ol_T$, such that $H_0'\cap D=S_2$ and $H_0''\cap D=S_3$ (see Step \ref{bundleZ}); we note that the divisor
 $H_0'+H_0''+D$
  is simple normal crossing in $Z$.

  By Step \ref{H0} and Rem.~\ref{relative}  we have $H_0'\sim H_0''\sim H_0$ and
$D \sim H_0-\ph^*(N)$, hence    $H_0'+H_0''+D\sim 3H_0-\ph^*(N)$.
 Moreover $-K_Z\sim\ph^*(-K_T-N)+3H_0$ and
  $$-K_Z+H_0\sim\ph^*(-K_T-N)+4H_0\sim A+H_0'+H_0''+D.$$
  By Norimatsu's Lemma \cite[Lemma 4.3.5]{lazI} we get
  $H^1(Z,H_0)=H^1(Z,K_Z-K_Z+H_0)=0$.
\end{proof}
\begin{step}
 $S_1$ is a complete intersection of $H_0$ and a divisor in $|2H_0|$.
\end{step}
\begin{proof}
  We note first of all that $H_0$ is a smooth $\pr^1$-bundle over $T$. Indeed, since $d=2$ by Step \ref{2}, 
  for every  $t\in T$  we have $\ol_Z(H_0)_{|Z_t}\cong\ol_{\pr^2}(1)$ (see Step \ref{H}). Note that no fiber $Z_t$ can be contained in $H_0$, because $H_0\cap S_2=\emptyset$ by Step \ref{disjoint}; therefore $(H_0)_{|Z_t}$ is a line for every $t\in T$, and $\ph_{|H_0}\colon H_0\to T$ is a $\pr^1$-bundle.

  Now we show that $\ol_{H_0}(S_1)\cong\ol_Z(2H_0)_{|H_0}$.
We consider  the divisor $H_0'$ as in the proof of Step \ref{vanishing2}, so that $H_0'\sim H_0$ and $H_0'\cap D=S_2$.
  
  We have that $D\cap H_0$ is $\{pt\}\times T$ in $D\cong\pr^1\times T$, hence $D\cap H_0\subset H_0$ is a section of $\ph$, disjoint from both $S_1$ and $H_0'$, because $D\cap S_1=\emptyset$ and $D\cap H_0'\cap H_0=S_2\cap H_0=\emptyset$ (see Steps \ref{disjoint} and \ref{D}).
Since $S_1$ has degree $2$ over $T$, in $H_0$ we have
$$S_1\sim_{H_0}2(H'_0)_{|H_0}+(\ph_{|H_0})^*(G),$$
where $G$ is a divisor in $T$. By restricting to the section $D\cap H_0$, we get $G\sim 0$ and $\ol_{H_0}(S_1)\cong\ol_Z(2H_0)_{|H_0}$.

Finally we consider the exact sequence in $Z$:
  $$0\la\ol_Z(H_0)\la\ol_Z(2H_0)\la\ol_Z(2H_0)_{|H_0}\la 0.$$
  By Step \ref{vanishing2} we have a surjection $H^0(Z,\ol_Z(2H_0))\twoheadrightarrow 
H^0(H_0,\ol_Z(2H_0)_{|H_0})$, and this yields the statement.
\end{proof}
\begin{step} \label{final}
  $X\to Z\stackrel{\ph}{\to} T$ is as in Construction B (Prop.~\ref{B}).
\end{step}
\begin{proof}
We just note that $-K_T\pm N$ is ample by Lemma \ref{fanoB}, and $\delta_T\leq 3$ by Lemma \ref{deltaB}.
\end{proof}
\section{Geometry of the del Pezzo fibration $\sigma\colon X\to T$}\label{sec6}
\noindent 
Let $X$ be a smooth Fano variety with $\delta_X=3$. By Th.~\ref{main} there are a smooth Fano variety $T$, and a morphism $\sigma\colon X\to T$, as in 
 Construction A or B; let us denote $X$ by $X_A$ in the former case, and by $X_B$ in the latter.

In this section we study the geometry of the fibration in del Pezzo surfaces $\sigma$, starting with the description of its fibers in \S\ref{properties}. Then, in \S\ref{relA} and \S\ref{relB}, we describe the 4-dimensional cone $\NE(\sigma)$ and the associated relative contractions, respectively for $X_A$ and $X_B$. In particular this allows to determine whether the same $\sigma\colon X\to T$ can be obtained with Construction A or B in different ways, or for different choices of divisors on $T$.

Finally in \S\ref{cond} 
we show that, if $X\not\cong F\times T$ ($F$ the blow-up of $\pr^2$ at 3 non-collinear points), then $T$ must satisfy some non-trivial condition.

We keep the same notation as in Sections \ref{first} and \ref{second}; in particular we recall that $\sigma$ factors as  $X \xrightarrow{h} Z \xrightarrow{\varphi} T$, where
${\varphi}$ is a $\pr^2$-bundle, and ${h}$ is the blow-up along 3 horizontal, pairwise disjoint, codimension 2 smooth subvarieties $S_1,S_2,S_3\subset Z$.  
\subsection{Fibers of $\sigma$}\label{properties}
Let $X$ be as above, and $X_t$ the fiber of $\sigma$ over $t\in T$.

For $X_A$, the morphism $\sigma$ is smooth, and $X_t$ is isomorphic to the blow-up of $\pr^2$ at non-collinear 3 points for every $t\in T$.

For $X_B$,  recall that 
$\varphi_{|S_1} \colon S_1 \to T$ is a finite cover of degree 2; let $\Delta\subset T$ be the branch divisor (note that $\Delta\neq\emptyset$).
Then $\sigma$ is smooth  over $T\smallsetminus\Delta$, where
$X_t$ is isomophic to the blow-up of $\pr^2$ at 4 points in general linear position. Instead, for $t\in\Delta$, $X_t$ is singular and isomorphic to the blow-up of $\pr^2$ at 3 non-collinear points, one of which is a double point contained in a line (see Rem.~\ref{rem}); thus $X_t$ is irreducible and has one rational double point of type $A_1$ (see for instance \cite[\S IV.2.3]{EH}).
\begin{remark}\label{toric}
  Let $X$ be a smooth Fano variety with $\delta_X=3$.
Then $X$ is toric if and only if  $X$ is obtained using Construction A from a toric Fano variety $T$.
  
Indeed, if $X$ is toric, then $\sigma\colon X\to T$ is smooth (so that $X=X_A$) and $T$ is toric.

Conversely, if we apply Construction A to a toric Fano variety $T$, then $Z$ is toric, and the three sections $S_1,S_2,S_3\subset Z$ are invariant for the torus action, so that $X$ is toric too.
\end{remark}
\subsection{ The cone $\NE(\sigma)$ and relative  contractions for $X_A$}\label{relA}
\begin{remark}\label{NE A}
  Consider $X=X_A$, let $X_{t}$ be a fiber of $\sigma$, and  $\iota\colon X_{t} \hookrightarrow X$  the inclusion.
 Since $\sigma$ is smooth, it follows from \cite[Prop.~1.3]{wisndef} that the pushforward $\iota_*\colon \mathcal N_1(X_{t}) \to \mathcal N_1(X)$ yields an isomorphism among $\N(X_{t})$ and $\ker\sigma_*$, and among the cones $\NE(X_{t})$ and $\NE(\sigma)$. Moreover every relative elementary contraction of $X/T$ restricts to an elementary contraction of $X_{t}$, and viceversa.

  Since $X_{t}$ is the blow-up of $\pr^2$ at 3 non-collinear points $p_1, p_2, p_3$, the cone $\NE(X_{t})$ is generated by the classes of the six $(-1)$-curves of $X_{t}$, given by the exceptional curve
  $e_i$ over $p_i$, and  the transform $e_i'$ of the line
  $\overline{p_jp_k}$.  In $X$ we have $e_i=E_i\cap X_{t}$ and $e_i'=E_i'\cap X_{t}$;   the exceptional divisors $E_i,E_i'$ are $\pr^1$-bundles over $T$.
  
  Figure \ref{figuraconoA} shows the 3-dimensional polytope obtained as a hyperplane section of the $4$-dimensional cone $\NE(\sigma)$.
\begin{figure}
\includegraphics[width=0.3\columnwidth]{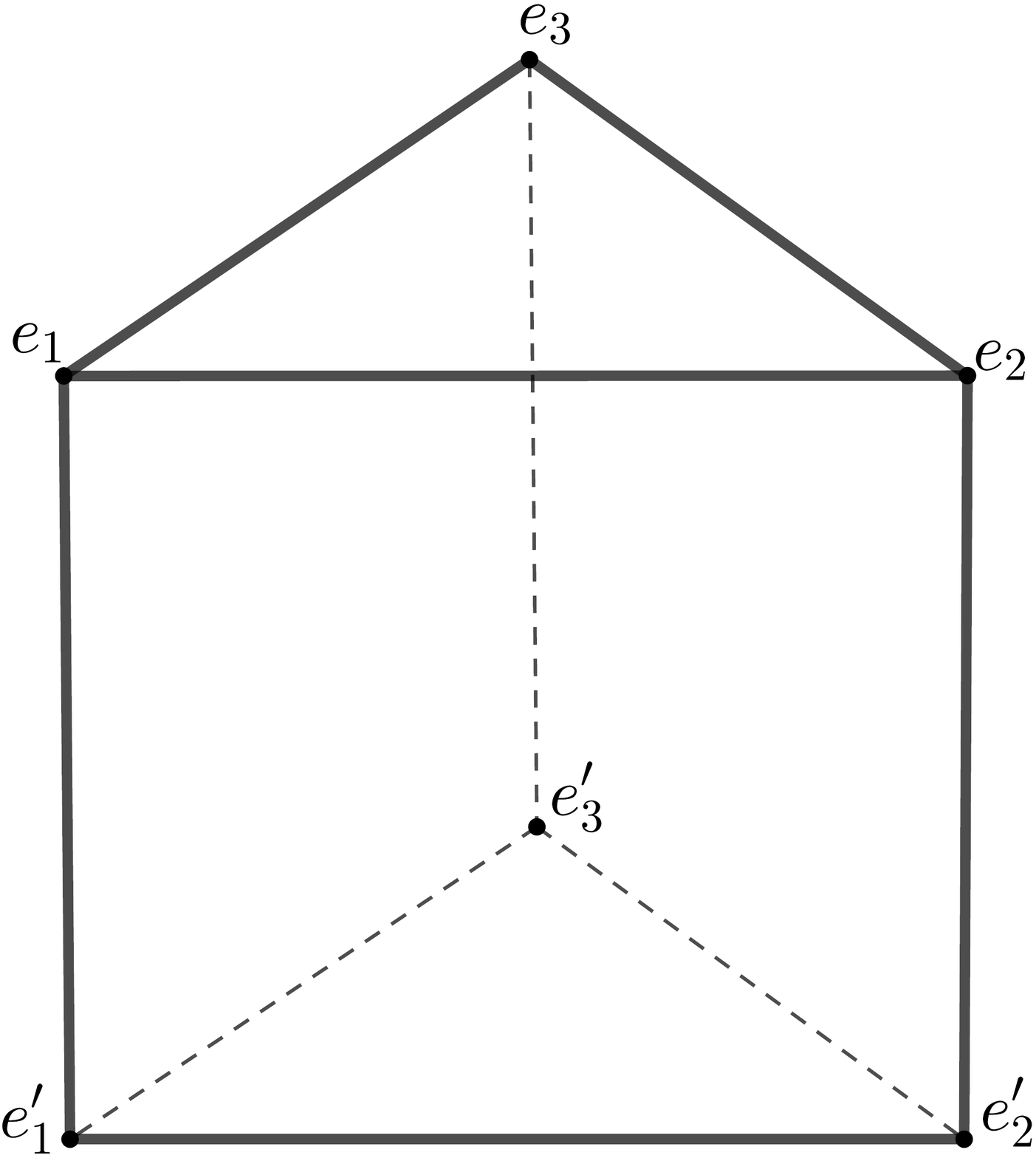}
\caption{A section of $\NE(\sigma)$ for $X_A$}\label{figuraconoA}
\end{figure}
\end{remark}
\begin{lemma}
\label{X_A}
Let $X=X_A$. Then there is a commutative diagram:
$$\xymatrix{ X\ar[r]^-{\wi{h}}\ar[d]_h\ar[dr]_{\sigma}& 
  {\wi Z=\pr_T(\ol(-D_1)\oplus\ol(-D_2)\oplus\ol(-D_3))}\ar[d]^{\wi{\ph}}\\
  Z\ar[r]_\ph&T}$$
where $\Exc(\wi{h})=E_1'\cup E_2'\cup E_3'$, so that performing
Construction A from $T$ with divisors $D_1,D_2,D_3$, or with divisors $-D_1,-D_2,-D_3$,
yields the same Fano variety $X$.
\end{lemma}
\begin{proof}
  By Rem.~\ref{NE A} the face of $\NE(\sigma)$ spanned by the classes of $e_1',e_2', e_3'$ yields a  contraction $\wi{h}\colon X\to\wi{Z}$ such that
 $\widehat \ph\colon\wi{Z}\to T$ is a $\pr^2$-fibration, and $\wi{h}$ is the blow-up of three sections that are fibrewise in general linear position. Moreover by Lemmas \ref{E} and \ref{E'} we have $E_i\cong \pr_T(\ol(-D_j)\oplus\ol(-D_k))$ and  $E_i' \cong \pr_T(\ol(D_j)\oplus\ol(D_k))$ for every $i,j,k$ with $\{i,j,k\}=\{1,2,3\}$; this implies the statement.
\end{proof}
Observe that the  birational contractions $h$ and $\wi h$ correspond to the two simplicial facets of $\NE(\sigma)$, see Fig.~\ref{figuraconoA}. Instead the three non-simplicial facets yield conic bundles $X\to Y$, where $Y$ is a $\pr^1$-bundle over $T$.
\subsection{ The cone $\NE(\sigma)$ and relative  contractions for $X_B$}\label{relB}
\begin{lemma}
\label{relative cone}
Let  $X=X_B$, $X_{t_0}$ a smooth fiber of $\sigma$, and $\iota\colon X_{t_0}\hookrightarrow 
X$ the inclusion.
  Then every relative elementary contraction of $X/T$ restricts to a non-trivial contraction of $X_{t_0}$, and
$\iota_*\NE(X_{t_0})=\NE(\sigma)$.
\end{lemma}
\begin{proof}
  Clearly $\iota_*\NE(X_{t_0})\subseteq\NE(\sigma)$. For the converse,
let $R$ be an extremal ray of $\NE(\sigma)$, and
   $c_R\colon X\to X_R$ the associated elementary contraction of $X$. We show that $X_{t_0}$ must contain some curve contracted by $c_R$; this implies that $R$ is in the image of  $\NE(X_{t_0})$ via $\iota_*$, so that $\iota_*\NE(X_{t_0})=\NE(\sigma)$.

 The statement is clear if $c_R$ is of fiber type, so let us assume that $c_R$ is birational.
 Every fiber $X_{t}$ of $\sigma$ is irreducible (see \S\ref{properties}), and $c_R(X_{t})\subset X_R$ is the fiber of $X_R\to T$ over $t$. Since $\dim X_R=n$, we have $\dim c_R(X_t)=2$ and hence $X_t\not\subset\Exc(c_R)$.

 On the other hand every fiber of $c_R$ is contained in some $X_t$, and we conclude that every fiber of $c_R$ has dimension $\leq 1$.
 By \cite[Thm.~1.2]{wisn}  we have $\dim\Exc(c_R)= n-1$, therefore $\dim\sigma(\Exc(c_R))=n-2$, $\sigma(\Exc(c_R))=T$, and $\Exc(c_R)$ meets every fiber $X_t$.
\end{proof}
\begin{remark}
\label{NE B}
Let  $X=X_B$; we use Lemma \ref{relative cone} to describe the cone $\NE(\sigma)$.

Let $X_{t}$ be a smooth fiber of $\sigma$ and $\{p_1,p_1', p_2, p_3\} \in Z_{t}$ be the  points blown-up by $h_{|X_{t}}\colon X_{t}\to Z_{t}$, where $p_i= S_i \cap Z_{t}$ for $i=2,3$, and $\{p_1, p_1'\}=S_1 \cap Z_{t}$. The $5$-dimensional cone  $\NE(X_{t})$ is generated by the classes of the ten $(-1)$-curves in $X_{t}$, given by the exceptional curves and the transforms of the lines through two blown-up points. 
We denote by
$e_i$ (respectively $e_1'$) the exceptional curve over $p_i$ (respectively $p_1'$), and $\ell_{i,j}$
 (respectively $\ell_{1,1'}$, $\ell_{1',i}$ for $i=2,3$)
the transform of the line $\overline{p_ip_j}$
(respectively $\overline{p_1p_1'}$, $\overline{p_1'p_i}$ for $i=2,3$). 

Consider as above $\iota\colon X_{t}\hookrightarrow X$ and $\iota_*\colon\N(X_{t})\twoheadrightarrow\ker\sigma_*$.
The kernel of $\iota_*$ has dimension 1, and is generated by the numerical class in $\mathcal N_1(X_{t})$ of the 1-cycle $(e_1-e_1')$. By looking at the numerical relations in $\mathcal N_1(X_{t})$, we provide a complete description of the cone $\NE(\sigma)$: Figure \ref{figuraconoB} shows the $3$-dimensional polytope obtained as a hyperplane section of the $4$-dimensional cone $\NE(\sigma)$. We see that $\NE(\sigma)$
 has $7$ extremal rays, and the figure shows their generators.
\begin{figure}
\includegraphics[width=0.5\columnwidth]{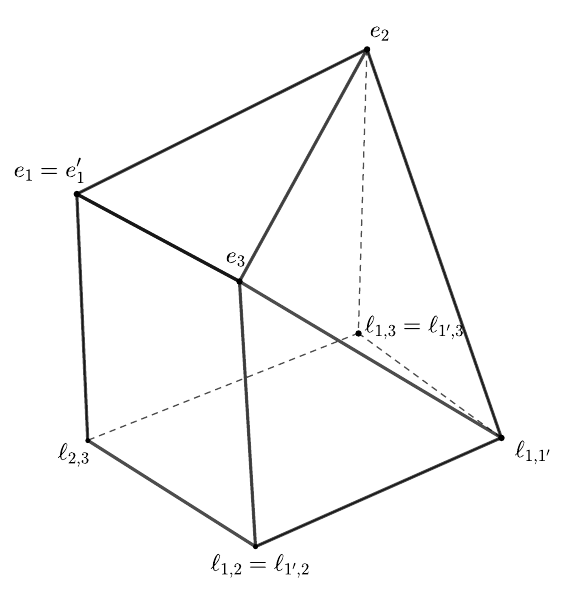}
\caption{A section of $\NE(\sigma)$ for $X_B$}\label{figuraconoB}
\end{figure}

It follows from \cite[Thm.~1.2]{wisn} that every relative elementary contraction of $\NE(\sigma)$ is the blow-up of a smooth variety along a smooth codimension 2 subvariety. The contraction corresponding to  $[e_1]=[e_1']$ (respectively $[e_2], [e_3]$) is the blow-down of $E_1$ (respectively $E_2, E_3$), while the contractions corresponding to $[\ell_{1,1'}]$ and $[\ell_{2,3}]$ have respectively exceptional divisors  $\widetilde H_0$ and $\widetilde D$ (the strict transforms of $H_0$ and $D$ from $Z$).
Lastly, for $i=2,3$ we denote by $G_i$ the exceptional divisor of the contraction corresponding to $[\ell_{1,i}]=[\ell_{1',i}]$.
\end{remark}
Observe that $\NE(\sigma)$ has 4 simplicial facets and 3 non-simplicial facets; $\NE(h)$ is the facet generated by $[e_1],[e_2],[e_3]$.
\begin{proposition}
\label{X_B}
Let $X=X_B$, and consider the facet $\langle [\ell_{1,2}],[\ell_{1,1'}],[e_3]\rangle$ (respectively $\langle [\ell_{1,3}],[\ell_{1,1'}],[e_2]\rangle$)  of $\NE(\sigma)$. The associated contraction $\wi h$ yields a commutative diagram: 
$$\xymatrix{ X\ar[r]^-{\wi{h}}\ar[d]_h\ar[dr]^{\sigma}& 
  {Z}\ar[d]^{\ph}\\
  Z\ar[r]_\ph&T}$$
 where $\Exc(\widehat h)=G_2\cup \widetilde H_0\cup E_3$ (respectively  $G_3\cup \widetilde H_0\cup E_2$),
and $X\stackrel{\wi h}{\to}Z\stackrel{\ph}{\to}T$ is as described in Construction B. 
\end{proposition}
\begin{proof}
  We consider the face $\langle [\ell_{1,2}],[\ell_{1,1'}],[e_3]\rangle$, the other case being analogous. Let $\wi h\colon X\to\wi{Z}$ be the associated contraction, and $\wi\ph\colon\wi{Z}\to T$ the map over $T$, so that $\wi\ph$ is an elementary contraction and $\sigma=\wi\ph\circ\wi h$ . 

  By Rem.~\ref{NE B}, $\wi Z$ is smooth and $\wi h$ is the blow-up of three pairwise disjoint smooth, codimension 2, irreducible subvarieties, with  $\Exc(\widehat h)= G_2\cup\widetilde H_0\cup E_3$.

  We show that $\widehat \ph\colon \wi Z\to T$ is a $\pr^2$-bundle. First note that all the fibers $\widehat Z_t$ of $\widehat \varphi$ are surfaces, and the general one is isomorphic to $\pr^2$.
  Let $A:= \widehat h(E_2)$, and  $\widehat Z_t$ be general, so that $\ol_{\widehat Z}(A)_{|\widehat Z_t} \cong \ol_{\pr^2}(1)$ and $A \cdot \ell =1$ for a line $\ell\subset \wi Z_t$.
  Since $\NE(\widehat \varphi)$ has dimension 1,  the relative Kleiman's criterion  implies that $A$ is $\widehat \varphi$-ample, thus there exists an ample divisor $M$ on $T$ such that $A':= A + m \widehat \varphi^{\,*}(M)$ is ample on $\widehat Z$ for $m \gg 0$. Note that $\ol_{\widehat Z}(A')_{|\widehat Z_t} \cong \ol_{\pr^2}(1)$, so we apply \cite[Prop.~3.2.1]{beltrametti_sommese} and deduce that $\widehat Z=\pr_T(\ma{G})$ for some rank 3 vector bundle $\ma{G}$ on $T$.
  
 By construction $\widehat S_3:=\widehat h(E_3)$ is a section of $\widehat \varphi$, and moreover $\ma{N}_{\widehat S_3/{\widehat Z}}^{\vee} \cong\ma{N}_{S_3/Z}^{\vee} \cong \ol(N) \oplus \ol$. Up to tensoring with some line bundle on $T$, we may assume that the section $\widehat S_3$ gives an exact sequence on $T$:
$$0 \la \ol(N) \oplus \ol \la \ma{G} \la \ol \la 0.$$
Since $-K_T+N$ is ample (see Lemma \ref{fanoB}), we have $h^1(T,N)=0$ and the above sequence splits, hence $\wi{Z}\cong\pr_T(\ol(N)\oplus\ol\oplus\ol)=Z$.

Finally it is not difficult to see that $\widehat  h(\widetilde H_0)$ is  a section of $\varphi$ corresponding to the projection onto $\ol_T$, and that $\widehat h(G_2)$ is a complete intersection of elements in $|H|$ and $|2H|$.
\end{proof}
\begin{remark}
\label{quadricbundle}
Let $X=X_B$, and consider the facet $\langle [e_2],[e_3],[\ell_{1,1'}]\rangle$  of $\NE(\sigma)$ (see Fig.~\ref{figuraconoB}). The associated contraction $h'$ yields a commutative  diagram
$$\xymatrix{ X\ar[r]^-{h'}\ar[d]_h\ar[dr]_{\sigma}& 
  {W}\ar[d]^{\alpha}\\
  Z\ar[r]_\ph&T}$$
where $\Exc(h')= E_2\cup E_3\cup \w H_0$, $W$ is smooth,
and $\alpha\colon W\to T$ is a quadric bundle, singular over $\Delta\subset T$.

Instead the three non-simplicial facets of $\NE(\sigma)$ yield conic bundles $X\to Y$, where $Y$ is a $\pr^1$-bundle over $T$.

We also note that, among the seven exceptional divisors associated to the extremal rays of $\NE(\sigma)$, $E_2$, $E_3$, $\w{H}_0$, and $\w{D}$ are $\pr^1$-bundles over $T$, while  $E_1$, $G_2$, and $G_3$ are $\pr^1$-bundles 
over  $S_1$, the double cover of $T$ ramified along $\Delta$.
\end{remark}
\subsection{Conditions on $T$}\label{cond}
Given a smooth Fano variety $T$, Constructions A and B give algorithms to construct from $T$ a smooth Fano variety $X$ with $\dim X=\dim T+2$ and $\rho_X=\rho_T+4$. For every $T$ one can get $X_A=F\times T$, where $F$   is the blow-up of $\pr^2$ at 3 non-collinear points, and $\sigma\colon X\to T$ the projection; this corresponds to the choice $D_1\sim D_2\sim D_3$.
In order to  
get an $X$ different from $F\times T$, the variety $T$ must satisfy some conditions, as follows.
\begin{lemma}\label{conditions}
  Let $T$ be a smooth Fano variety.
  Suppose that there exists
  a Fano variety $X$ obtained from $T$ as in Construction A or B, and such that $\sigma\colon X\to T$ is not isomorphic to the projection $F\times T\to T$.

  Then  there exists a hyperplane $\Lambda\subset\N(T)$ containing all the classes $[C]$ where $C\subset T$ is a curve with $-K_T\cdot C=1$.

  In particular, every extremal ray of length 1 of $\NE(T)$ is contained in $\Lambda$, and $\NE(T)$ has at least one extremal ray $R$ with length $\ell(R)\geq 2$.
\end{lemma}
\begin{proof}
Let $C\subset T$ be a curve with $-K_T\cdot C=1$. 
  Suppose first that $X=X_A$.  
By Lemma \ref{fanoA} for every $i,j=1,2,3$ we have $1\leq (-K_T+D_i-D_j)\cdot C=1+(D_i-D_j)\cdot C$, so that $D_i\cdot C= D_j\cdot C$.
We exclude by assumption the case $D_1\sim D_2\sim D_3$, hence
for some $i\neq j$ we must have $D_i\not\sim D_j$; set $\Lambda:= (D_i-D_j)^{\perp}$, so that $\Lambda$ is a hyperplane in $\N(T)$, and contains $[C]$.

The case where $X=X_B$ is similar, we set $\Lambda:=N^{\perp}$.
\end{proof}
In the next section we apply these conditions to the case where $T$ is a del Pezzo surface.
\section{The case of dimension $4$}\label{4folds}
\noindent Let $X$ be a Fano $4$-fold with $\delta_X=3$. 
By Th.~\ref{main}, $X$ is obtained with Construction A or B  from a del Pezzo surface $T$ with $\rho_X-\rho_T=4$ and $\delta_T\leq 3$. Since $T$ is a surface, it is easy to see that $\delta_T=\rho_T-1$, thus $\rho_T\in\{1,2,3,4\}$ and $\rho_X\in\{5,6,7,8\}$.

If $\rho_T\geq 3$, then $\NE(T)$ is generated by classes of $(-1)$-curves, in particular there are no extremal rays of length $>1$. Then it follows from Lemma \ref{conditions}  that $X\cong F\times T$, where $F$ is the blow-up of $\pr^2$ in $3$ non-collinear points.

If instead $\rho_T=1$, then $\rho_X=5$; this case is completely classified in \cite[Th.~1.1]{delta3_4folds}, and there are $6$ families.

Finally suppose that $\rho_T=2$, so that $\rho_X=6$. If $X$ is toric,
after the classification of toric Fano $4$-folds by Batyrev \cite{bat2,sato}, we see that $X$ has combinatorial type $U$ (in the notation of \cite{bat2}), and there are $8$ possibilities for $X$. More precisely, we have $T\cong \mathbb{F}_1$ in the two cases $U_2$ and $U_4\cong F\times \mathbb{F}_1$, and $T\cong\pr^1\times\pr^1$ in the remaining six cases, including $U_5\cong  F\times \pr^1\times\pr^1$.

In the non-toric case we get the following result, that together with the previous discussion implies Prop.~\ref{classification}.
\begin{proposition}\label{newfamilies}
  There are
  three families of
 non-toric Fano $4$-folds $X$ with $\delta_X=3$ and $\rho_X=6$.
\end{proposition} 
  \begin{proof}
    By Th.~\ref{main} and Rem.~\ref{toric}, $X$ is obtained  with Construction B  from a del Pezzo surface $T$ with $\rho_T=\rho_X-4=2$, namely $T\cong\pr^1\times\pr^1$ or $T\cong\mathbb{F}_1$. The divisor $N$ on $T$ is such that the class of $N$ is effective and non-zero, and $-K_T\pm N$ is ample.

      It is not difficult to see that there are three choices of $N$ satisfying these conditions:
 $N\in|\ol_{\pr^1\times\pr^1}(0,1)|$,   $N\in|\ol_{\pr^1\times\pr^1}(1,1)|$, and $N$ the pullback of a general line in $\pr^2$ in the case of $\mathbb{F}_1$.

In all cases $N$ is nef, hence the tautological divisor $H$ of $Z=\pr_T(\ol(N)\oplus\ol\oplus\ol)$ is nef. On the other hand $Z$ is toric, thus $H$ is globally generated: we conclude that the general complete intersection of elements in $|H|$ and $|2H|$ is smooth, and we can apply Construction B.
  Finally we get
   three families of Fano $4$-folds, respectively $X_{B_0}$, $X_{B_1}$, and $X_{B_2}$.
 \end{proof}
 \noindent   {\bf Description of $X_{B_0}$.}
   We have $X_{B_0}\cong\pr^1\times Y$, where $Y$ is the Fano $3$-fold obtained by blowing-up $\pr^3$ along a line, a conic disjoint from the line, and two non-trivial fibers of the blow-up of the line. In fact $Y$ is obtained  with Construction B  from $T_Y=\pr^1$ with $N_Y\in|\ol_{\pr^1}(1)|$.

\medskip
 
 \noindent   {\bf Description of $X_{B_1}$.}
  Set $Z:=\pr_{\pr^1\times \pr^1}(\ol(1,1)\oplus\ol^{\oplus 2})\stackrel{\ph}{\to}\pr^1\times \pr^1$; the $4$-fold $X_{B_1}$ is the blow-up of $Z$ along three pairwise disjoint smooth surfaces $S_1,S_2,S_3$. The surfaces $S_2$ and $S_3$ are sections of $\ph$,  they are isomorphic to $\pr^1\times\pr^1$ and have normal bundle $\ol(-1,-1)\oplus\ol$. The blow-up $X_1$ of $Z$ along $S_2$ and $S_3$ is a toric Fano $4$-fold of combinatorial type $Q_{1}$, following \cite{bat2}.

  On the other hand the surface $S_1$ is a double cover of $\pr^1\times\pr^1$
  with $-K_{S_1}=\ph^*\ol_{\pr^1\times \pr^1}(1,1)_{|S_1}$ (see Rem.~\ref{doublecover}), so that $S_1$ is a del Pezzo surface of degree 4. 

\medskip

\noindent {\bf Description of $X_{B_2}$.}
Here $T=\mathbb{F}_1$ and  $N$ is the pullback of a general line in $\pr^2$. Then $Z$ is a $\pr^2$-bundle over $\mathbb{F}_1$, and $S_2,S_3$ are two sections; the blow-up $X_1$ of $Z$ along $S_2$ and $S_3$ is a toric Fano $4$-fold of combinatorial type $Q_2$, following \cite{bat2}. Moreover $\ph_{|S_1}\colon S_1\to\mathbb{F}_1$ is a double cover with $-K_{S_1}=\ph_{|S_1}^*(2N-E)$, where $E\subset\mathbb{F}_1$ is the exceptional curve, so that $S_1$ is a del Pezzo surface of degree $6$, the blow-up of $\pr^2$ at $3$ points.

\medskip

We give in Table \ref{table} some numerical invariants of the Fano $4$-folds 
$X_{B_0}$, $X_{B_1}$, and $X_{B_2}$;
they are computed using standard methods, see \cite[Lemmas 3.2 and 3.3]{delta3_4folds}. In the last column $T$ denotes the tangent bundle.
\begin{table}[!htbp]\caption{Numerical invariants}\label{table}
$$\begin{array}{|c|c|c|c|c|c|c|c||}
\hline\hline
\rule{0pt}{2.5ex}  \text{$4$-fold}   & b_3 & h^{2,2} & h^{1,3} & K^4 & K^2\cdot c_2 & h^0(-K) & \chi({T}) \\
\hline\hline

    X_{B_0} &  0 & 10 & 0& 224 & 152 & 51 & 4 \\

\hline
    
X_{B_1} &  0 & 14 & 0 &222 & 156 & 51 &  -2 \\

\hline

X_{B_2} & 0 & 12 & 0 & 223 &  154 & 51 &  1\\

\hline\hline
  \end{array}$$
\end{table}

\medskip

Finally we apply our results to the study of conic bundles on Fano $4$-folds.

Let $X$ be a Fano $4$-fold and 
$\eta\colon X\to Y$ a conic bundle such that $\rho_X-\rho_Y\geq 3$. Let us denote by $\bigtriangleup:=\{y\in Y| \eta^{-1}(y) \text{  is singular}\}$ the discriminant divisor.

If $X\cong S\times T$ is a product of two del Pezzo surfaces,  
then it follows easily 
that
$Y\cong \pr^1\times T$
and $\eta$ is induced by a conic bundle $S\to\pr^1$ (see for instance \cite[Lemma 2.10]{eleonora}); in particular all the connected components of $\bigtriangleup$ are isomorphic to $T$. 

Let us assume that $X$ is not a product of surfaces. Then we have 
$\rho_X-\rho_Y=\delta_X=3$
by \cite[Th.~4.2(1)]{eleonora} and \cite[Th.~1.1]{M_R}, and $\rho_X\in\{5,6\}$ by  Prop.~\ref{classification}, so that the possible $X$ are classified.
The case where $\rho_X=5$ has been studied in \cite[Cor.~2.18]{delta3_4folds}.
 As an application of our results, we describe
 the case $\rho_X=6$.
\begin{corollary} \label{cor:conic_bundles} Let $\eta\colon X\to Y$ be a conic bundle, where $X$ is a Fano $4$-fold with $\rho_X=6$, and $\rho_X-\rho_Y=3$. Let $\bigtriangleup$ be the discriminant divisor. Then one of the following hold:
\begin{enumerate}[$(i)$]
\item
$\bigtriangleup\cong \mathbb{F}_1\sqcup\mathbb{F}_1$ and either $X\cong U_2$ or $X\cong U_4$;
\item
  $\bigtriangleup\cong \pr^1\times\pr^1 \sqcup  \pr^1\times \pr^1$ and $X$ is isomorphic to one of the following varieties: $U_1$, $U_3$, $U_5$, $U_6$, $U_7$, $U_8$,
  $X_{B_0}$;
\item 
$\bigtriangleup\cong  \pr^1\times \pr^1\sqcup S_1$ with $S_1$ a del Pezzo surface of degree $4$ and $X\cong X_{B_1}$;
\item 
$\bigtriangleup\cong \mathbb{F}_1\sqcup S_1$ with $S_1$ a del Pezzo surface of degree $6$ and $X\cong X_{B_2}$. 
\end{enumerate}
\end{corollary}
\begin{proof}  
If $X$ is a product of two del Pezzo surfaces, the statement is easy, so let us assume this is not the case; then we have 
 $\delta_X=3$ by \cite[Th.~1.1]{M_R}.

  Now, replacing $\psi \circ f$ by $\eta$ in Th.~\ref{factorization}, we check that all the properties $(a)-(h)$ are satisfied. Indeed, by \cite[Prop.~3.5 (1)]{eleonora}, $\eta$ factors as a composition $X\to X_2\to Y$ where the first map satisfies $(f)$; and in view of \cite[Th.~4.2 (2)]{eleonora} also $(a)$, $(b)$, $(c)$, $(d)$ hold. Finally, $(g)$ follows by the proof of \cite[Prop.~3.4]{eleonora}, while $(h)$ is shown in Step 2 of the proof of \cite[Th.~4.2 (2)]{eleonora}. Therefore, we may run the arguments of the proof of Th. \ref{main} with $\eta$ instead of  $\psi \circ f$. We keep the notation as in that theorem. 

  Then $\bigtriangleup=A_1 \sqcup A_2$, and by Step \ref{Z} it follows that $A_i\cong B_i \cong S_i$ for $i=1,2$. Moreover, using Step \ref{sectionA} we know that $A_2\cong T$, where in our case either $T\cong \pr^1\times\pr^1$ or $T\cong \mathbb{F}_1$. If $A_1$ is a section of $\xi\colon Y\to T$ then using Steps \ref{d} and \ref{U} we get $(i)$ or $(ii)$. Otherwise, by Step \ref{final} we deduce that $X$ is obtained as in Construction B, and finally that
  $X$ is isomorphic to $X_{B_0}$, $X_{B_1}$, or $X_{B_2}$,
  following 
  the proof of Prop.~\ref{newfamilies}. As observed in the above descriptions of
    these non-toric families, $S_1$ is respectively $\pr^1\times\pr^1$, 
 a del Pezzo surface of degree 4, or a del Pezzo of degree 6, hence we get
$(ii)$, $(iii)$, or $(iv)$.
\end{proof}
\begin{remark} Let $X$ be a Fano 4-fold with $\delta_X=3$. By Th.~\ref{factorization}, we know that $\psi \circ f\colon X\to Y$ is a conic bundle with $\rho_X-\rho_Y=3$. All the possible targets $Y$ have been classified in \cite{M_R}, but in \cite[Prop.~1.2(b)]{M_R} there is a missing case, that is $Y\cong \mathbb{P}_{\mathbb{F}_1}(\mathcal{O} \oplus \mathcal{O}(N))$, with $N$ being the pullback of a general line through the blow-up $\mathbb{F}_1\to \mathbb{P}^2$. In this case, Construction A gives $X=U_2$; while performing Construction B we obtain the non-toric Fano family $X_{B_2}$.  
\end{remark}
\section{The case $\rho_X=5$}\label{last}
\noindent In this section we consider Fano varieties with 
$\delta_X=3$ and $\rho_X=5$, the minimal Picard number, and prove Prop.~\ref{rho5}, Cor.~\ref{rho2}, and Th.~\ref{n-1,1}.
By Th~\ref{main}, $X$ is obtained as in Construction A or B  from a smooth Fano variety $T$ with $\rho_T=1$. We denote by $\ol_T(1)$ the ample generator
of $\Pic(T) \cong\Z$, and $\ol_T(a):=\ol_T(1)^{\otimes a}$ for every $a\in\Z$. Recall that $\ol_T(-K_T) \cong \ol_T( i_T )$, where $i_T$ is the index of $T$.
\begin{proof}[Proof of Prop.~\ref{rho5}]
 We show the uniqueness of $\sigma\colon X\to T$.
Let us assume that there exists another $\bar \sigma \colon X \to \overline T$ as in Construction A or B.   We show that $\sigma$ and $\bar \sigma$ coincide up to an isomorphism $T \cong \overline T$.

Let $R$ be an extremal ray of the cone $\NE(\sigma)$.
By Remarks \ref{NE A}, \ref{NE B}, and \ref{quadricbundle}, the contraction associated to $R$ is the blow-up of a smooth subvariety $S$ which is an irreducible codimension $2$ complete intersection,  and 
either $S\cong T$ or $S\cong S_1$. Let $E\subset X$ be the exceptional divisor;
since the normal bundle of $S$ is decomposable, we have $E\cong\pr_S(\ol\oplus L)$ with $L\in\Pic(S)$.

We first assume that $n\geq 5$, where $n:=\dim X$.
Recall that $S_1$ is a ramified double cover of $T$, and $\rho_T=1$, so that the ramification divisor is ample;
since $\dim S_1=n-2> 2$, \cite{cornalba} yields $\rho_{S_1}=1$. Therefore
 in any case we have $\rho_S=1$ and $\rho_E=2$.

Since $\rho_E=2$, $E$ has at most two elementary contractions, one being the $\pr^1$-bundle $E\to S$.
 If $L\cong\ol_S$, then $E\cong\pr^1\times S$, and the second elementary contraction is $E\to\pr^1$. Otherwise we can assume that $L$ is ample; in this case $E$ has an elementary divisorial contraction sending a divisor to a point (see for instance \cite[p.~10768]{minimal}). 

Consider now the restriction $\bar \sigma_{|E} \colon E \to \bar \sigma(E) \subseteq \overline T$. The Stein factorization gives  $$E \xrightarrow{\psi} B \xrightarrow{\nu} \bar \sigma(E),$$ where $\psi$ is a contraction of $E$, and $\nu$ is finite. 
Thus $\dim B=\dim\bar \sigma(E) \leq\dim \overline T=n-2$; on the other hand every fiber of $\bar\sigma$ is a surface, thus $\dim B\geq\dim E-2=n-3\geq 2$.
By the previous observation on the possible contractions of $E$, we deduce that $B\cong S$ and $\psi$ is the $\pr^1$-bundle $E\to S$,
 so that $\bar\sigma$ contracts the fibers of $E\to S$ and $R\subset\NE(\bar \sigma)$. 

This holds for every extremal ray $R$ of $\NE(\sigma)$, so that 
  $\NE(\sigma) \subseteq \NE(\bar \sigma)$. Since these two cones are both 4-dimensional  faces of $\NE(X)$, we conclude that they are the same, and  by \cite[Prop.\ 1.14]{debarreUT} $\sigma$ and $\bar \sigma$ coincide up to an isomorphism $T \cong \overline T$.

  \medskip

  Suppose now that $n=4$, so that $T\cong\overline T\cong\pr^2$.
  We repeat the same argument as above with an extremal ray $R$ such that $E$ is a $\pr^1$-bundle over $T$. We still have $\dim B\geq \dim E-2=1$, and if $\dim B=1$, then $E\cong\pr^1\times\pr^2$ and $\psi$ the projection onto   $\pr^1$. Thus the general fiber of $\bar\sigma$ over $\bar\sigma(E)$ is isomorphic to $\pr^2$, which is impossible by the description of the fibers in \S\ref{properties}. We conclude that $\dim B\geq 2$, and as before that $R\subset\NE(\bar\sigma)$. If $X=X_A$ we get $\NE(\sigma) \subseteq \NE(\bar \sigma)$ and hence $\NE(\sigma) = \NE(\bar \sigma)$. If $X=X_B$, we see from Rem.~\ref{NE B} and Fig.~\ref{figuraconoB} that the extremal rays $R$ of  $\NE(\sigma)$ such that $E$ is a $\pr^1$-bundle over $T$ generate a $4$-dimensional cone. We conclude that $\dim(\NE(\sigma)\cap\NE(\bar\sigma))=4=\dim\NE(\sigma)=\dim\NE(\bar\sigma)$, so that again the two cones coincide.

  Finally suppose that $n=3$, so that $T\cong\overline T\cong\pr^1$, and repeat the same argument as above. 
 We conclude that either $R\subset\NE(\bar\sigma)$, or $E$ has two different $\pr^1$-bundle structures, hence  $E\cong\pr^1\times\pr^1$ with normal bundle $\ol(-1,-1)$. It is not difficult to see that there is at most one extremal ray $R$ with this property, so we conclude that $\NE(\sigma)= \NE(\bar \sigma)$ as in the case $n=4$.
 
\medskip

Let $X=X_A$.  By Rem.~\ref{NE A} and Lemma \ref{X_A}, $\sigma\colon X\to T$ has exactly two different factorizations as in Construction A, through 
$\pr_T(\ol(D_1)  \oplus  \ol(D_2) \oplus \ol(D_3))$ and $\pr_T(\ol(-D_1)  \oplus  \ol(-D_2) \oplus \ol(-D_3))$. Moreover $\ol_T(D_1),  \ol_T(D_2),
\ol_T(D_3)$ are determined up to 
reordering  and  tensoring with a line bundle. 

We may assume that $\ol_T(D_1)\cong \ol_T(a)$, $\ol_T(D_2) \cong \ol_T$, and $\ol_T(D_3) \cong \ol_T(b)$, with $a,b \in \mathbb Z$. By Lemma \ref{fanoA}, we see that $(a,b)$ must satisfy the following conditions: 
\begin{equation} \label{8.0}
|a|,\ |b|,\ |a-b| \leq i_T-1.
\end{equation}
 We say that two pairs of integers $(a,b)$ and $(a',b')$ are equivalent (denoted by $\sim$) if they satisfy \eqref{8.0} and give isomorphic $X_A$'s. 
By the previous discussion, 
we see  that all the pairs equivalent to $(a,b)$ can be obtained by the following  relations:
\begin{equation*}
(a,b) \sim (b,a),\ (-b,a-b),\ (-a,-b).
\end{equation*}
Hence up to equivalence we can subsequently assume:
\begin{enumerate}[$\bullet$]
\item $a \geq 0$: indeed if $a < 0$, we replace $(a,b)$ with $(-a,-b)$;
\item $a\geq b$: if $a<b$, we replace $(a,b)$ with  $(b,a)$;
\item $b \leq 0$: if $b >0$, we replace $(a,b)$ with $(a-b, -b)$;
\item $a \geq -b$: if $a < -b$, we replace $(a,b)$ with $(-b,-a)$,
\end{enumerate}
so that in the end:
 $b\leq 0$ and $a\geq |b|$. 
 These conditions, together with \eqref{8.0}, are equivalent to the conditions in the statement:
\begin{equation}\label{8.1}
b\leq 0,\quad |b| \leq \frac{i_T-1}{2},\quad\text{and}\quad |b| \leq a \leq i_T-1-|b|.
\end{equation}
Finally
it is not difficult to see that if $(a,b),(a',b')$ satisfy \eqref{8.1} and 
 $(a,b)\sim(a',b')$, then $a=a'$ and $b=b'$.

\medskip

Let $X=X_B$. 
By Rem.~\ref{NE B}, Prop.~\ref{X_B}, and Rem.~\ref{quadricbundle}, $\sigma\colon X\to T$ has exactly three different factorizations as in Construction B, all through 
$Z=\pr_T(\ol(N) \oplus  \ol \oplus \ol)$, where $N \not \equiv 0$ and $2N$ is effective. Therefore $\ol_T(N)$ is ample and isomorphic to $\ol_T(a)$ for some integer $a \geq 1$. By Lemma \ref{fanoB} we see that $a$ must satisfy $a \leq i_T-1$. 
\end{proof}
\begin{proof}[Proof of Cor.~\ref{rho2}]
  We have $\delta_X\geq\codim\N(D,X)=\rho_X-2$. If $\rho_X\geq 6$, then $\delta_X\geq 4$, so that by Th.~\ref{delta4} we have $X\cong S\times T$ where $S$ is a del Pezzo surface with $\rho_S=\delta_X+1\geq 5$. If $D_T\subset T$ is a prime divisor then $\dim\N(S\times D_T,X)\geq\rho_S$, therefore $D$ must dominate $T$ under the projection $\pi_T\colon S\times T\to T$. Moreover
  $\pi_T$ is not finite on $D$, so that
when we consider the pushforward
  $(\pi_T)_*\colon\N(S\times T)\to\N(T)$,  we have $\ker(\pi_T)_*\cap\N(D,X)\neq\{0\}$, and we conclude that $\rho_T=1$.

Suppose that $\rho_X=5$, and note that  the assumptions imply that $\dim X\geq 3$. If $\delta_X>\rho_X-2$, then $\delta_X=\rho_X-1$, which means that $X$ contains a prime divisor $D'$ with $\dim\N(D',X)=1$.  
By \cite[Lemma 3.1]{minimal} this would imply that $\rho_X\leq 3$, a contradiction. Therefore 
   $\delta_X=\rho_X-2=3$, and $X$ is as in Prop.~\ref{rho5}.
\end{proof}
\begin{proof}[Proof of Th.~\ref{n-1,1}] $(ii)\Rightarrow(i)\ $
In $Z=\pr_T(\ol(a)\oplus\ol\oplus\ol)$ let $S_2,S_3$ be the sections of $Z\to T$ corresponding to the projections onto the trivial summands, and $F_1:=\pr_T(\ol\oplus\ol)\hookrightarrow Z$ given by the projection $\ol(a)\oplus\ol\oplus\ol\twoheadrightarrow \ol\oplus\ol$ (this is $D$ in the notation of Construction B), so that $S_2,S_3\subset F_1$. Then, in both Constructions A and B, $X$ is the blow-up of $Z$ along $S_2$, $S_3$, and a third subvariety $S_1$ disjoint from $F_1$.
  
If $E_1'\subset X$ is the transform of $F_1$, we have  
 $E_1'\cong\pr^1\times T$ with normal bundle $\ma{N}_{E_1'/X}\cong \pi_{\pr_1}^*\ol(-1)\otimes\pi_T^*\ol(-a)$.  Let $\Gamma_0\subset T$ be a curve, and let $\Gamma\subset E_1'$ be the curve corresponding to $\{pt\}\times\Gamma_0$. 
We show that $[\Gamma]$ generates an extremal ray of $\NE(X)$; the associated contraction $\tau\colon X\to X'$ is elementary divisorial with $\Exc(\tau)=E_1'$ and $\tau(\Exc(\tau))\cong\pr^1$.

Set $c:=\ol_T(1)\cdot\Gamma_0$, and consider $e_1'\subset E_1'$ corresponding to $\pr^1\times\{pt\}$. We have:
$$E_1'\cdot e_1'=-1,\ -K_X\cdot e_1'=1,\ E_1'\cdot\Gamma=-ac,\ -K_X\cdot\Gamma=c(i_T-a).$$

If $2a>i_T$, consider $H:=a(-K_X)+(i_T-a)E_1'$. Then $H\cdot C\geq 0$ for every curve $C\subset X$ not contained in $E_1'$. Moreover $H\cdot e_1'=2a-i_T>0$ and $H\cdot \Gamma=0$, so that $H$ is nef and $H^{\perp}\cap\NE(X)=\R_{\geq 0}[\Gamma]$.

If instead $2a\leq i_T$, consider the exceptional divisor $E_2\cong\pr_T(\ol\oplus\ol(a))$ over $S_2$, and $e_2\subset E_2$ a fiber of the $\pr^1$-bundle $E_2\to T$. We have: 
$$E_2\cdot e_2=-1,\ E_2\cdot e_1'=1,\ E_2\cdot\Gamma=0,\ -K_X\cdot e_2=1,\ E_1'\cdot e_2=1,$$
and $E_2\cdot R\geq 0$ for every extremal ray $R$ of $\NE(X)$ not containing $[e_2]$. Consider $H':=a(-K_X)+(i_T-a)E_1'+(i_T-2a+1)E_2$. We have $H'\cdot e_1'=1$, $H'\cdot \Gamma=0$, $H'\cdot e_2=2a-1>0$, and $H'\cdot R>0$ for every extremal ray of $\NE(X)$ not containing $[e_1'],[\Gamma],[e_2]$. We conclude that $H'$ is nef and $(H')^{\perp}\cap\NE(X) =\R_{\geq 0}[\Gamma]$.

\medskip

$(i)\Rightarrow(ii)\ $  If $n=4$ the statement is shown in \cite[Cor.~1.4]{delta3_4folds}, so that we can assume $n\geq 5$.

Consider the pushforward $\tau_*\colon\N(X)\to\N(X')$. We have $\tau_*(\N(\Exc(\tau),X))=\R[\tau(\Exc(\tau))]$, thus  $\dim\N(\Exc(\tau),X)=2$,
and it follows from Cor.~\ref{rho2} that 
$X$ is as described in Prop.~\ref{rho5}. We note first of all that $X\not\cong F\times T$ ($F$ the blow-up of $\pr^2$ at three non-collinear points), as $F\times T$ has no elementary divisorial contraction sending a divisor to a curve, hence the case $X=X_A$ with $a=b=0$ is excluded, and $i_T>1$. 

Let $\Gamma\subset X$ be a curve contracted by $\tau$. We show that at least one 
of the exceptional divisors of the extremal rays of $\NE(\sigma)$ has non-zero intersection with $\Gamma$. We keep the same notation as in Rem.~\ref{relative} and Sections \ref{first} and \ref{second}.

By contradiction, suppose otherwise: then the classes of these divisors in $\Nu(X)$ all belong to the hyperplane $\Gamma^{\perp}$. This in turn implies in $\Nu(Z)$ that, if $X=X_A$ (respectively $X=X_B$), the linear span of the classes of $F_1,F_2,F_3$  (respectively the classes of $D$ and $H_0$) has dimension $1$. This is possible only if $X=X_A$ and $a=b=0$, that we have already excluded.

Hence there exists some  extremal ray of $\NE(\sigma)$ such that the associated exceptional divisor  $E$ satisfies $E\cdot\Gamma\neq 0$. We show that $E=\Exc(\tau)$. 

With the same notation as in the proof of Prop.~\ref{rho5}, $E$ is a $\pr^1$-bundle over $S$, where $S\cong T$ or $S\cong S_1$, and $\rho_S=1$.
If $T_0\subset\Exc(\tau)$ is a non-trivial fiber of $\tau$, we have $E\cap T_0\neq\emptyset$ and hence $\dim(E\cap T_0)\geq n-3>0$, and the Stein factorization
of $\tau_{|E}$  induces a non-trivial contraction  $\psi\colon E\to B$. Since $E$ meets every non-trivial fiber $T_0$ of $\tau$, $\psi(\Exc(\psi))$ is a curve, and by the analysis of the elementary contractions of $E$ in the proof of Prop.~\ref{rho5} we conclude that $E\cong\pr^1\times S$ and $\psi$ is the projection onto $\pr^1$. This means that $\dim\tau(E)=1$, hence $E=\Exc(\tau)$.

If $X=X_A$, it is not difficult to check that $\pr^1\times T$ appears among the divisors $E_1,E_2,E_3,E_1',E_2',E_3'$ if and only if $b=0$, so that $Z=\pr_T(\ol(a)\oplus\ol\oplus\ol)$. Then we have $E_1\cong E_1'\cong\pr^1\times T$ with normal bundles $\ma{N}_{E_1/X}\cong \pi_{\pr_1}^*\ol(-1)\otimes\pi_T^*\ol(a)$ and $\ma{N}_{E_1'/X}\cong \pi_{\pr_1}^*\ol(-1)\otimes\pi_T^*\ol(-a)$. Since $a>0$, we have $\Exc(\tau)=E_1'$, and we get the statement.

If $X=X_B$,
we show that $S\cong T$. Indeed, if  $S\cong S_1$, $E$ should be one of the divisors $E_1,G_2,G_3$ (see Rem.~\ref{NE B}), and by Prop.~\ref{X_B} we have $E_1\cong G_2\cong G_3\cong \pr_{S_1}(\ol\oplus\ol(H_{|S_1}))$. On the other hand  $H_{|S_1}$ is  linearly equivalent to the ramification divisor of the non-trivial double cover $\ph_{|S_1}\colon S_1\to T$ by Rem.~\ref{doublecover}, so that $H_{|S_1}\not\sim 0$. Therefore we get $S\cong T$, and it is not difficult to check that $\Exc(\tau)=\w{D}$.
\end{proof}
\providecommand{\noop}[1]{}
\providecommand{\bysame}{\leavevmode\hbox to3em{\hrulefill}\thinspace}
\providecommand{\MR}{\relax\ifhmode\unskip\space\fi MR }
\providecommand{\MRhref}[2]{%
  \href{http://www.ams.org/mathscinet-getitem?mr=#1}{#2}
}
\providecommand{\href}[2]{#2}

\end{document}